\documentclass[12pt]{amsart} 
\usepackage{amsmath,centernot,amssymb,leftindex}

\usepackage[shortlabels]{enumitem}
\setlist[enumerate,1]{label={(\roman*)}}
\usepackage{mathtools}
\usepackage[style=numeric,sortcites]{biblatex}
\AtEveryBibitem{\clearfield{url}}
\usepackage[hidelinks]{hyperref}
\newcommand{\dens}{\mathrm{d}}
\newcommand{\udens}{\bar{\mathrm{d}}}
\newcommand{\ldens}{\underline{\mathrm{d}}}

\newcommand{\mb}{\mathbf}

\newtheorem{theorem}{Theorem}
\newtheorem{lemma}[theorem]{Lemma}
\newtheorem{proposition}[theorem]{Proposition}
\newtheorem{corollary}[theorem]{Corollary}
\newtheorem*{claim}{Claim}
\numberwithin{theorem}{section}
\numberwithin{equation}{section}

\theoremstyle{definition}

\newtheorem{observation}[theorem]{Observation}
\newtheorem{definition}[theorem]{Definition}

\newtheorem{remark}[theorem]{Remark}

\newtheorem{conjecture}[theorem]{Conjecture}
\newtheorem{problem}{Problem}

\title[Kneser- and Jin-type inverse theorems]{Kneser- and Jin-type inverse theorems in discrete abelian groups}

\date{\today}

\author{John T. Griesmer}
\email{jtgriesmer@gmail.com}
\address{Department of Applied Mathematics and Statistics, Colorado School of Mines, Golden, Colorado}
\usepackage{amstext}


\begin{document}

	\begin{abstract} We characterize the pairs of sets $A, B$ in an arbitrary (countable or uncountable) discrete abelian group $\Gamma$ satisfying $\tilde{m}(A+B)<\tilde{m}(A)+\tilde{m}(B)$, where $\tilde{m}$ is an arbitrary finitely additive translation-invariant probability measure on $\Gamma$, extending M.~Kneser's theorem on Haar measure in compact abelian groups. 
		
	We then characterize, for an arbitrary F{\o}lner sequence or F{\o}lner net $\mb F=(F_{i})_{i\in I}$ on $\Gamma$, those $A$, $B$ satisfying $\ldens_{\mb F}(A+B)<\ldens_{\mb F}(A)+\ldens_{\mb F}(B)$, where $\ldens_{\mb F}(C):=\liminf_{i\in I} |C\cap F_{i}|/|F_{i}|$. This extends Kneser's theorem on  lower asymptotic density in $\mathbb N$.  
	
	We also generalize theorems of Prerna Bihani and Renling Jin by characterizing pairs $A$, $B$ satisfying $\dens^{*}(A+B)<\dens^{*}(A)+\dens^{*}(B)$, where $\dens^{*}$ is upper Banach density on $\Gamma$.

	\end{abstract}
	
	\maketitle

	\section{Introduction}
	
	\subsection{Kneser's theorem for lower asymptotic density}
	
	For $A$, $B\subseteq \mathbb Z$, their \emph{sumset} is $A+B:=\{a+b:a\in A, b\in B\}$.  The \emph{lower asymptotic density}  of a set $C\subseteq \mathbb Z$ is
	\[\ldens(C):=\liminf_{n\to\infty} \tfrac{1}{n}|C\cap \{1,\dots,n\}|.\]	
	We write $\dens(C)$ if the limit exists.  M.~Kneser classified, for each $r\in \mathbb N$, the $r$-tuples of sets $A_{1},\dots, A_{r}\subseteq \mathbb Z$ satisfying $\ldens(A_{1}+A_{2}+\cdots +A_{r})< \sum_{j=1}^{r} \ldens(A_{j})$.  We state the classification for the case $r=2$.
	\begin{theorem}[{\cite[pp.~463--464]{Kneser_AsymptoticDensity}}, Dichtesatz f{\"u}r die asymptotiche Dichte]\label{th:KneserLAD}
		Let $A$, $B\subseteq \mathbb Z$ be nonempty sets satisfying 
		\begin{equation}\label{eq:LADstrictHypothesis}
			\ldens(A+B)<\ldens(A)+\ldens(B).
		\end{equation}
		  Then there is a $k\in \mathbb N$ such that 
		\begin{enumerate}
			\item\label{item:DensPeriodic} $\dens(A+B)=\dens(A+B+k\mathbb Z)$;
			\item\label{item:CDequation} $\dens(A+B)=\dens(A+k\mathbb Z)+\dens(B+k\mathbb Z)-\frac{1}{k}$;
			\item\label{item:Cofinite} $A+B$ is cofinite in $(A+B+k\mathbb Z)\cap \mathbb N $.			
		\end{enumerate}	
	\end{theorem}
See \S7-\S10 of Chapter I in \cite{Halberstam_Roth_Sequences} for another proof of Theorem \ref{th:KneserLAD}.

Assertion \ref{item:DensPeriodic} in Theorem \ref{th:KneserLAD} reduces the analysis of pairs satisfying (\ref{eq:LADstrictHypothesis}) to the problem of classifying subsets $C$, $D$ of a finite abelian group $G$ satisfying $|C+D|<|C|+|D|$.  Assertion \ref{item:CDequation} further specializes to those $C$ and $D$ satisfying $|C+D|=|C|+|D|-1$.  Such pairs were classified by Kemperman \cite{Kemperman_SmallSumsets}; see \cite{Grynkiewicz_Quasiperodic_Kemperman}, \cite{Lev_CriticalPairs}, or Chapters 6 and 9 of \cite{Grynkiewicz_Structural} for discussion.

Kneser's \cite[Satz 1]{Kneser_SummenmengenLokalkompakten} (stated as Theorem \ref{th:SummenmengenSatz1} below) is an analogue of Theorem \ref{th:KneserLAD} where $\mathbb Z$ is replaced by a locally compact abelian group $G$ and $\ldens$ is replaced by Haar measure.  From this he deduces Satz 6 and Satz 6a of \cite{Kneser_SummenmengenLokalkompakten}, partially generalizing Theorem \ref{th:KneserLAD} to discrete abelian groups $\Gamma$.  Noting that Satz 6 and Satz 6a apply only to a very restricted class of subsets $A$, $B$ of $\Gamma$, he remarks\footnote{cf.~the remarks at the end of \cite[Section 1]{Kneser_SummenmengenLokalkompakten}.  Translated by the present author: 
 
	``The class of almost periodic subsets of $\mathbb Z^{d}$ is very small, however, and  the result of Satz 6a is of interest insofar as only a few results on these questions in several dimensions have been available so far -- to my knowledge, none at all for the asymptotic density considered here.  Cf. {\sc{Kasch}}, {\sc F}.: \emph{Wesentliche Komponenten bei Gitterpunktmengen. Erscheint demn{\"a}chst im J.~reine angew.~ Math.}~and the literature cited there.''}  that no analogue of Theorem \ref{th:KneserLAD} in known for lower asymptotic density in $\mathbb Z^{d}$ for any $d\geq 2$, and suggests two problems.
	\begin{problem}\label{prob:LAD}
		For $d\geq 2$, characterize the sets $A$, $B\subseteq \mathbb Z^{d}$ satisfying $\ldens(A+B)<\ldens(A)+\ldens(B)$.  
	\end{problem}
	
	 \cite{Kneser_SummenmengenLokalkompakten} does not explicitly define $\ldens$ for subsets of $\mathbb Z^{d}$, and there are several natural ways to do so. Our results below will cover any such definition. 
		
	\begin{problem}\label{prob:Mean}Given a discrete abelian group $\Gamma$ and a finitely additive, translation-invariant measure $\tilde{m}$ defined on $\mathcal P(\Gamma)$, characterize the sets  $A$, $B\subseteq \Gamma$ satisfying $\tilde{m}(A+B)<\tilde{m}(A) + \tilde{m}(B)$.
	\end{problem}
	\cite{Kneser_SummenmengenLokalkompakten} uses ``$\delta$,'' rather than ``$\tilde{m}$,''  to denote a finitely additive, translation-invariant probability measure, and calls such a measure a ``density.''

	    Addressing Problem \ref{prob:LAD}, our Theorem \ref{th:MainFolnerStrict} below is an analogue of Theorem \ref{th:KneserLAD} which applies to lower asymptotic  density in $\mathbb Z^{d}$, and more generally to lower density along a F{\o}lner sequence or F{\o}lner net in a discrete abelian group.  In \S\ref{sec:Extending} we state Theorem \ref{th:FolnerZ}, a special case of Theorem \ref{th:MainFolnerStrict},  and explain why its conclusion is necessarily weaker than the conclusion of Theorem \ref{th:KneserLAD}.
	
	Theorem \ref{th:MainOneMeanStrict} below addresses Problem \ref{prob:Mean}.  For the sake of exposition we begin by generalizing to other densities in $\mathbb Z$.

\subsection{Extending Theorem \ref{th:KneserLAD} in the integers.}\label{sec:Extending} A natural generalization of lower asymptotic density replaces the sequence of intervals $\{1,\dots,n\}$ with a sequence $\mb F=(F_{n})_{n\in \mathbb N}$ of finite subsets of $\mathbb Z$ satisfying $\lim_{n\to \infty} \frac{|(F_{n}+t)\triangle F_{n}|}{|F_{n}|}=0$ for all $t\in \mathbb Z$.  Such an $\mb F$ is called a \emph{F{\o}lner sequence}.  For $A\subseteq \mathbb Z$, the \emph{lower density of $A$ along} $\mb F$ is defined as $\ldens_{\mb F}(A):=\liminf_{n\to\infty} |A\cap F_{n}|/|F_{n}|$, and the \emph{upper density of $A$ along $\mb F$} is $\udens_{\mb F}(A):=\limsup_{n\to\infty} |A\cap F_{n}|/|F_{n}|$.  We write $\dens_{\mb F}(A)$ if the limit exists.

Our aim is to extend Theorem \ref{th:KneserLAD} by replacing $\ldens$ with $\ldens_{\mb F}$ for an arbitrary F{\o}lner sequence $\mb F$.  Conclusion \ref{item:Cofinite} of Theorem \ref{th:KneserLAD} cannot be recovered in this generality.  For example, setting $F_{n}=\{2^{2^{n}}+1, 2^{2^{n}}+2,\dots,2\cdot 2^{2^{n}}\}$ and $A=B=\bigcup_{n=1}^{\infty} F_{n}$, we have $\dens_{\mb F}(A)=\dens_{\mb F}(B)=\dens_{\mb F}(A+B)=1$. For each $k\in \mathbb N$ we have $A+B+k\mathbb Z=\mathbb Z$, and $A+B$ is not cofinite in $(A+B+k\mathbb Z)\cap \mathbb N$.

Assuming $\ldens_{\mb F}(A+B)<\ldens_{\mb F}(A)+\ldens_{\mb F}(B)$, the analogue of conclusion \ref{item:CDequation} is easy to obtain once one finds \emph{some} $k\in \mathbb N$ satisfying $\dens_{\mb F}(A+B)=\dens_{\mb F}(A+B+k\mathbb Z)$; see \S\ref{sec:RecoverKEQ} for details.  

Thus we abandon hope for \ref{item:Cofinite} and postpone our interest in \ref{item:CDequation}.  Our proposed generalization of Theorem \ref{th:KneserLAD} is now this appealing, but false, conjecture:
\begin{conjecture}\label{conj:TooStrong}
	If $A$, $B\subseteq \mathbb Z$ are nonempty and $\mb F$ is a F{\o}lner sequence satisfying $\ldens_{\mb F}(A+B)<\ldens_{\mb F}(A)+\ldens_{\mb F}(B)$, then there is a $k\in \mathbb N$ such that $\dens_{\mb F}(A+B)=\dens_{\mb F}(A+B+k\mathbb Z)$.
\end{conjecture}
\begin{remark}\label{rem:Example} Conjecture \ref{conj:TooStrong} fails for $F_{n}=\{1,2,\dots, b_{n}\}$ with $b_{n+1}/b_{n}\to \infty$.  With such $\mb F$, let $A=B=\bigcup_{n\in \mathbb N}\{\lfloor b_{n}/2\rfloor,\dots, b_{n}\}$.  Then $\dens_{\mb F}(A)=\dens_{\mb F}(B)=\dens_{\mb F}(A+B)=1/2$, but for every $k\in \mathbb Z$, we have $\dens_{\mb F}(A+B)<\dens_{\mb F}(A+B+k\mathbb Z)=1$. However, with $\Psi_{n}=\{\lfloor b_{n}/2\rfloor,\dots,b_{n}\}$, we have $1=\dens_{\mb \Psi}(A+B)<\dens_{\mb \Psi}(A)+\dens_{\mb \Psi}(B)=2$ and $\dens_{\mb \Psi}(A+B)=\dens_{\mb \Psi}(A+B+\mathbb Z)$.  Thus we recover the hypothesis and the conclusion of Conjecture \ref{conj:TooStrong} by passing to subsets of $F_{n}$ which form a new F{\o}lner sequence.  Our first result says  this can be done in general, possibly after passing to a subsequence of $\mb F$.
\end{remark}

\begin{theorem}\label{th:FolnerZ}
	Let $A$, $B\subseteq \mathbb Z$ be nonempty.  Then there is a $k\in \mathbb N$ (depending only on $A+B$) such that if $\mb F=(F_{n})_{n\in \mathbb N}$ is a F{\o}lner sequence satisfying 
\begin{equation}\label{eq:ZDeltaHyp}
	\delta:= \ldens_{\mb F}(A)+\ldens_{\mb F}(B)-\ldens_{\mb F}(A+B)>0,
\end{equation}
	then there is a subsequence $(F_{n_{j}})_{j\in \mathbb N}$ and sets $\Psi_{j}\subseteq F_{n_{j}}$ such that $(\Psi_{j})_{j\in \mathbb N}$ is a F{\o}lner sequence satisfying 
	\begin{enumerate}
	\item\label{eq:FolnerBound}$\liminf_{j\to\infty} |\Psi_{j}|/|F_{n_{j}}|\geq k\delta$;
	\item\label{eq:Special1} 
	$\dens_{\mb \Psi}(A+B)=\dens_{\mb \Psi}(A+B+k\mathbb Z).$
	\item\label{eq:ReflectGap} $\dens_{\mb \Psi}(A)+\dens_{\mb \Psi}(B)-\dens_{\mb \Psi}(A+B)\geq \delta$;
\end{enumerate}
\end{theorem}
\begin{remark}
Conclusion \ref{eq:ReflectGap} says that $A$ and $B$ are still quite large with respect to $\dens_{\mb \Psi}$.  While foregoing \ref{eq:ReflectGap} would simplify the proof somewhat, we include it with an eye toward future applications.
\end{remark}

Theorem \ref{th:FolnerZ} is a special case of Theorem \ref{th:MainFolnerStrict} below, where the ambient group $\mathbb Z$ is replaced by a discrete abelian group and $\mb F$ is an arbitrary F{\o}lner sequence or F{\o}lner net.

\begin{remark}
The hypothesis (\ref{eq:ZDeltaHyp}) in Theorem \ref{th:FolnerZ} can be weakened to
\begin{equation}\label{eq:WeakZDeltaHyp}
\delta:=	\limsup_{n\to\infty} |A\cap F_{n}|/|F_{n}| + |B\cap F_{n}|/|F_{n}| - |(A+B)\cap F_{n}|/|F_{n}|>0
\end{equation}
without altering the conclusion: assuming (\ref{eq:WeakZDeltaHyp}), a subsequence of $\mb F$ satisfies (\ref{eq:ZDeltaHyp}).
\end{remark}

\begin{remark} Theorem \ref{th:FolnerZ} does not recover Theorem \ref{th:KneserLAD}: Theorem \ref{th:KneserLAD} uses the F{\o}lner sequence $F_{n}=\{1,\dots,n\}$ and does not require passing to subsets or to a subsequence.  Nor does Theorem \ref{th:KneserLAD} immediately imply Theorem \ref{th:FolnerZ}, even when $\mb F$ is a subsequence of $(\{1,\dots,n\})_{n\in \mathbb N}$. Possibly $\dens_{\mb F}(A+B)<\dens_{\mb F}(A)+\dens_{\mb F}(B)$, while $\ldens(A+B)\geq \ldens(A)+\ldens(B)$.  Such an example is given by increasing sequences  $a_{n}, b_{n}\in \mathbb N$ with $a_{n+1}/b_{n}\to 3$, $b_{n}/a_{n}\to \infty$, $F_{n} = \{1,\dots,b_{n}\}$, and $A=B=\bigcup_{n \in \mathbb N} [a_{n},b_{n}]$.  Here $\dens_{\mb F}(A)=\dens_{\mb F}(B)=\dens_{\mb F}(A+B)= 1$, while $\ldens(A)=\ldens(B)=1/3$, and $\ldens(A+B)= 2/3$.
	
	It would be interesting to prove Theorem \ref{th:KneserLAD} using Theorem \ref{th:FolnerZ} as a building block.
\end{remark}

\subsection{Upper asymptotic density is different}  
The \emph{upper asymptotic density} of a set $A\subseteq \mathbb Z$ is $\udens(A):=\limsup_{n\to\infty} \frac{1}{n}|A\cap \{1,\dots,n\}|$. 
Jin \cite{Jin_InverseProblem1, Jin_InverseProblem2,Jin_InverseProblemSolution} proves that if $A\subseteq \mathbb N$ satisfies $\gcd(A-\min(A))=1$ and $\udens(A)<1/2$, then $\udens(A+A)\geq \tfrac{3}{2}\udens(A)$.  Theorem 1.3 of \cite{Jin_InverseProblemSolution} classifies the sets satisfying $\udens(A+A)= \tfrac{3}{2}\udens(A)$ and the preceding hypotheses. Examples very different from those satisfying $\underbar{d}(A+A)<\underbar{d}(A)+\underbar{d}(A)$ are given:    $A\subseteq \mathbb N$ may satisfy  $0<\udens(A+A)=\tfrac{3}{2}\udens(A)<1$, while $A\cap (n+k\mathbb Z)\neq \varnothing$ for every $n, k\in \mathbb N$.  Bordes \cite{Bordes_sumsets} classifies sets $A\subseteq \mathbb Z$ where $\udens(A)$ is small and $\udens(A+A)<\tfrac{5}{3}\udens(A)$.

Our proofs say nothing about distinct pairs $A, B\subseteq \mathbb N$ satisfying $\udens(A+B)<\udens(A)+\udens(B)$.  For $A=B$ we get the following corollary of Theorem \ref{th:FolnerZ}, with a weaker conclusion than the results mentioned above.
\begin{corollary}\label{cor:UD}
	If $A\subseteq \mathbb Z$ satisfies $\delta:=2\udens(A)-\udens(A+A)>0$, then there is a $k\in \mathbb N$ and a F{\o}lner sequence $\Psi_{n}\subseteq \mathbb N$ such that \begin{enumerate}[(i)]
		\item $\liminf_{n\to\infty}|\Psi_{n}|/\max \Psi_{n} \geq k\delta$;
		\item $2\dens_{\mb \Psi}(A)-\dens_{\mb \Psi}(A+A)\geq \delta$;
		\item $\dens_{\mb \Psi}(A+A)=\dens_{\mb \Psi}(A+A+k\mathbb Z)$.
	\end{enumerate}
\end{corollary}
\begin{proof}
	Fix $A$ as in the hypothesis, and let $F_{n}=\{1,\dots,b_{n}\}$ satisfy  $\udens(A)=\dens_{\mb F}(A)$.  Note that $\dens_{\mb F}(A+A)\leq \udens(A+A)$, so $\dens_{\mb F}(A)+\dens_{\mb F}(A)-\dens_{\mb F}(A+A)\geq \delta$.  Now apply Theorem \ref{th:FolnerZ} to get a subsequence $(b_{n_{j}})_{j\in \mathbb N}$ and a F{\o}lner sequence $\Psi_{j}\subseteq \{1,\dots, b_{n_{j}}\}$ satisfying $\dens_{\mb \Psi}(A+A)=\dens_{\mb \Psi}(A+A+k\mathbb Z)$, $2\dens_{\mb \Psi}(A)-\dens_{\mb \Psi}(A+A)\geq \delta$, and $\lim_{j\to\infty}|\Psi_{j}|/b_{n_{j}}\geq k\delta$.
\end{proof}

\subsection{Upper Banach density in \texorpdfstring{$\mathbb Z$}{Z}}\label{sec:IntroUBD}

The \emph{upper Banach density} of a set $A\subseteq \mathbb Z$ is
\begin{equation}\label{eq:Zubd}
	\dens^{*}(A):=\sup\{\udens_{\mb F}(A):\mb F \text{ is a F{\o}lner sequence}\}.
\end{equation}

Bihani and Jin \cite[Theorem 1.1]{BihaniJin_Kneser} classified the sets $A\subseteq \mathbb N$ satisfying $\dens^{*}(A+A)<2\dens^{*}(A)$.  Jin \cite[Theorem 1.4]{Jin_CharacterizingStructure} classified the pairs $A, B\subseteq \mathbb N$ satisfying $\dens^{*}(A+B)<\dens^{*}(A)+\dens^{*}(B)$.  In \cite[Theorem 1.8]{Griesmer_SmallSumPairs} we extended the latter classification to subsets of countable abelian groups.  See Definition \ref{def:UBD} for the definition of upper Banach density in discrete abelian groups.  Here is the relevant part of \cite{Griesmer_SmallSumPairs}, Theorem 1.8.\footnote{\cite{BihaniJin_Kneser}, \cite{Jin_CharacterizingStructure}, and \cite{Griesmer_SmallSumPairs} provide additional detail.}
\begin{theorem}\label{th:CountableUBD}
	If $\Gamma$ is a countable abelian group and $A$, $B\subseteq \Gamma$ are nonempty sets satisfying $\dens^{*}(A+B)<\dens^{*}(A)+\dens^{*}(B)$, then there is a finite index subgroup $K\leq \Gamma$ satisfying $\dens^{*}(A+B)=\dens^{*}(A+B+K).$
\end{theorem}

Our next result generalizes Theorem \ref{th:CountableUBD} to discrete abelian groups of any cardinality. 
\begin{theorem}\label{th:MainUBDstrict}
	If $\Gamma$ is a discrete abelian group and $A$, $B\subseteq \Gamma$ are nonempty sets satisfying $\dens^{*}(A+B)<\dens^{*}(A)+\dens^{*}(B)$, then there is a finite index subgroup $K\leq \Gamma$ satisfying $\dens^{*}(A+B)=\dens^{*}(A+B+K).$
\end{theorem}
Theorem \ref{th:MainUBDstrict} will be proved at the end of \S\ref{sec:UBD}, as a consequence of Theorem \ref{th:TwoMeansStrict}.

\subsection{A potential extension}

The examples in Appendix C of \cite{BjorklundFish_ApproxInvariance} show that the conclusion of Theorem \ref{th:FolnerZ} cannot be recovered under the hypothesis $\dens_{\mb F}(A+B)< \dens^{*}(A)+\dens_{\mb F}(B)$, with the F{\o}lner sequence $F_{n}=\{-n,\dots,n\}$ in $\mathbb Z$.  To wit: there are nonempty sets $A$, $B\subseteq \mathbb Z$ satisfying $\dens_{\mb F}(A+B)<\dens^{*}(A)+\dens_{\mb F}(B)<1$, while $\dens_{\mb F}(A+B+k\mathbb Z)=1$ for every $k\in \mathbb N$.  In fact, this example satisfies $\dens^{*}((A+B)\cap (n+k\mathbb Z))<\dens^{*}(n+k\mathbb Z)$ for every $k, n\in \mathbb N$, so the approximate periodic structure guaranteed by Theorem \ref{th:FolnerZ} is absent.

\subsection{Compact groups}\label{sec:HaarMeasure}

The next theorem, also due to Kneser, is the analogue of Theorem \ref{th:KneserLAD} in the setting of locally compact abelian (LCA) groups.   If $K$ is an LCA group with Haar measure $\mu$, the corresponding \emph{inner Haar measure} $\mu_{*}$ is given by $\mu_{*}(C):=\sup\{\mu(E):E\subseteq C \text{ is compact}\}$.  

\begin{definition}\label{def:Stabilizer}
	Let $\Gamma$ be an abelian group and $C\subseteq \Gamma$. 	The \emph{stabilizer of} $C$ is the subgroup  $H(C):=\{\gamma \in\Gamma: C+\gamma=C\}$.
\end{definition}
Note that $H(C)$ satisfies $C+H(C)=C$.

\begin{theorem}[\cite{Kneser_SummenmengenLokalkompakten}, Satz 1]\label{th:SummenmengenSatz1} Let $G$ be an LCA group with Haar measure $\mu$.  If $A$, $B\subseteq G$ are nonempty $\mu$-measurable sets satisfying $\mu_*(A+B)<\mu(A)+\mu(B)$, then the stabilizer $K:=H(A+B)$ is compact, open, and satisfies
\begin{equation}\label{eq:KneserCritical}
	\mu(A+B)=\mu(A+K)+\mu(B+K)-\mu(K).
\end{equation}
\end{theorem}
As in \cite{Jin_CharacterizingStructure}, \cite{Griesmer_InverseTheorem}, \cite{Griesmer_SmallSumPairs}, and \cite{Griesmer_Semicontinuity}, Equation (\ref{eq:KneserCritical}) is crucial for our main argument.

Note that $K$ being open implies $G/K$ is a discrete abelian group.  Now Theorem \ref{th:SummenmengenSatz1} and the identity $A+B=A+B+K$ reduce the problem of classifying sets $A$, $B\subseteq G$ satisfying $\mu_{*}(A+B)<\mu(A)+\mu(B)$ to the corresponding problem for discrete abelian groups.  In discrete groups Haar measure may be taken to be counting measure, so the latter inequality becomes $|A+B|<|A|+|B|$.  

To emphasize our priorities, we restate the special case of Theorem \ref{th:SummenmengenSatz1} where $G$ is compact.
\begin{theorem}\label{th:Satz1Compact}
Let $G$ be a compact abelian group with Haar probability measure $\mu$.  If $A$, $B\subseteq G$ are nonempty $\mu$-measurable sets satisfying $\mu_{*}(A+B)<\mu(A)+\mu(B)$, then there is a compact open (hence finite index) subgroup $K\leq G$ satisfying the following:
\begin{enumerate}
	\item\label{item:Satz1FiniteIndex} $A+B=A+B+K$;
	\item\label{item:Satz1Stabilizer} $K=H(A+B)$;
	\item\label{item:Satz1Equation} $\mu(A+B)=\mu(A+K)+\mu(B+K)-\mu(K)$.
\end{enumerate}
\end{theorem}
Our third main result is Theorem \ref{th:MainOneMeanStrict}, which we state in \S\ref{sec:Means}.  It generalizes  Theorem \ref{th:Satz1Compact} by replacing $G$ with a discrete abelian group $\Gamma$ and $\mu$ with a finitely additive, translation-invariant probability measure defined all subsets of $\Gamma$.  Such measures are usually discussed in terms of the associated linear functional on $\ell^{\infty}(\Gamma)$, called an \emph{invariant mean}.  Thus Theorem \ref{th:MainOneMeanStrict} will apply to groups such as $\mathbb R$, $\mathbb T^{d}\times \mathbb Z^{r}$, etc., with no measurability hypothesis on the summands $A$, $B$.  It will also apply to ultraproducts\footnote{The setting we have in mind here is one where $(G_{n})_{n\in \mathbb N}$ is a sequence of abelian groups, $m_{n}$ is an invariant mean on $\ell^{\infty}(G_{n})$, $\mathcal U$ is a nonprincipal ultrafilter on $\mathbb N$, $\mb G=\prod_{n\to\mathcal U} G_{n}$ and we define a mean $m$ on $\ell^{\infty}(G)$ by first defining $m$ on uniformly closed subalgebra $\mathcal A\subseteq \ell^{\infty}(G)$ generated by the standard parts of internal functions $\mb f:=\prod_{n\to \mathcal U}f_{n}$, $f_{n}:G_{n}\to \mathbb C$, $m(\operatorname{st}(\mb f)):=\lim_{n\to \mathcal U} m_{n}(f_{n})$, then extending via Hahn-Banach.} of abelian groups, providing a key step in extending the results of \cite{Griesmer_Semicontinuity} beyond ultraproducts of compact abelian groups.

\subsection{Outline}

In \S\ref{sec:Means} and \S\ref{sec:FolnerDef} we introduce terminology and state our main results.  Section \ref{sec:Models} summarizes the main result of \cite{Griesmer_DiscreteSumsets}, which models sumsets in discrete groups by sumsets in compact groups, allowing us to exploit Theorem \ref{th:SummenmengenSatz1}.  Section \ref{sec:UBDproof} provides some technical lemmas and proves Theorem \ref{th:TwoMeansStrict}.

Section \ref{sec:Choquet} recalls the Choquet-Bishop-de Leeuw decomposition theorem for compact convex subsets of locally convex spaces.  Section \ref{sec:Hilbert} contains technical lemmas regarding such decompositions, Hilbert spaces associated to invariant means, and F{\o}lner nets in discrete abelian groups.  Section \ref{sec:MainProofs} proves Theorem \ref{th:MainOneMeanStrict} and Proposition \ref{prop:FullFolnerStrict}.  Section \ref{sec:LevelProof} proves a technical lemma stated in \S\ref{sec:Hilbert}.

\section{Means}\label{sec:Means}

Let $\Gamma$ be a discrete abelian group. If $A$ and $B$ are subsets of $\Gamma$, their \emph{sumset} is $A+B:=\{a+b:a\in A, b\in B\}$. 

We write $\ell^{\infty}(\Gamma)$ for the Banach space of bounded functions $f:\Gamma\to \mathbb C$, equipped with the supremum norm: $\|f\|_{\infty}:=\sup_{x\in \Gamma}|f(x)|$.  We write $\ell^{\infty}(\Gamma)^{*}$ for the (Banach space) dual of $\ell^{\infty}(\Gamma)$, equipped with the weak$^{*}$ topology.  For $A\subseteq \Gamma$, we write $1_{A}$ for the indicator function: $1_{A}(\gamma)=1$ if $\gamma\in A$, $1_{A}(\gamma)=0$ otherwise.  Thus $1_{\Gamma}$ is the constant function with value $1$.

A \emph{mean} on $\ell^\infty(\Gamma)$ is a linear functional $m:\ell^\infty(\Gamma)\to \mathbb C$ satisfying $m(1_\Gamma)=1$ and $m(f)\geq 0$ for all $f:\Gamma\to[0,1]$. We write $\mathcal M(\Gamma)$ for the set of means on $\ell^{\infty}(\Gamma)$. It is easy to verify that $\mathcal M(\Gamma)$ is a weak$^{*}$-compact convex subset of $\ell^{\infty}(\Gamma)^{*}$.

Associated to each mean $m$ is a finitely additive probability measure $\tilde{m}$ on $\mathcal P(\Gamma)$, given by $\tilde{m}(A):=m(1_A)$.  We will abuse notation and write $m(A)$ in place of $\tilde{m}(A)$.

We write $A\subset_m B$ if $m(A\setminus B)=0$, and we write $A\sim_{m} B$ if $m(A\triangle B)=0$.

If $f$ is a function on $\Gamma$ and $\gamma\in \Gamma$, the \emph{translate} $\tau_{\gamma}f$ is given by $(\tau_{\gamma}f)(x):=f(x-\gamma)$.
An \emph{invariant mean on $\ell^\infty(\Gamma)$} is a mean $m$ satisfying $m(\tau_{\gamma}f)=m(f)$ for all $f\in \ell^\infty(\Gamma)$ and all $\gamma\in \Gamma$.
Every discrete abelian group admits at least one invariant mean; this is often seen as a consequence of applying the Markov-Kakutani fixed-point theorem to the action $\tau^{*}$ on  $\mathcal M(\Gamma)$  given by translation: $(\tau_{\gamma}^{*}m)(f):=m(\tau_{\gamma}f)$.  For details see \cite{Paterson_Amenability}, specifically Proposition 0.15 of Chapter 0 therein.

The set of invariant means on $\ell^\infty(\Gamma)$ will be denoted $\mathcal M_{\tau}(\Gamma)$.  It is a nonempty weak$^*$-compact convex subset of $\ell^\infty(\Gamma)^*$. The Krein-Milman theorem therefore implies that its set of extreme points, denoted  $\mathcal M_{\tau}^{ext}(\Gamma)$, is nonempty.  An element of $\mathcal M_{\tau}^{ext}(\Gamma)$ is called an \emph{extreme invariant mean}.

\begin{definition}\label{def:UBD}
	If $\Gamma$ is a discrete abelian group and $A\subseteq \Gamma$, the \emph{upper Banach density} of $A$ is $\dens^{*}(A):=\sup\{m(A):m\in \mathcal M_{\tau}(\Gamma)\}$.
\end{definition}
Lemma \ref{lem:UBD} shows that Definition \ref{def:UBD} agrees with (\ref{eq:Zubd}).  The next lemma says that upper Banach density is realized by an extreme invariant mean.  See \cite[Observation 1.3]{Griesmer_DiscreteSumsets} for a proof.

\begin{lemma}
Let $\Gamma$ be a discrete abelian group and $A\subseteq \Gamma$.  There is a $\nu \in \mathcal M_{\tau}^{ext}(\Gamma)$ such that $\nu(A)=\dens^{*}(A)$. 
\end{lemma}

\begin{lemma}\label{lem:FiniteIndexMean}
	Let $\Gamma$ be a discrete abelian group and $K\leq \Gamma$ a finite index subgroup with index $k$.  If $C\subseteq \Gamma$ is a union of cosets of $K$, then for every $m\in \mathcal M_{\tau}(\Gamma)$ 
	\begin{equation}\label{eq:FIdensity}
		m(C)=\dens^{*}(C)=|\tilde{C}|/k,
	\end{equation}
	where $\tilde{C}$ is the image of $C$ in the quotient $\Gamma/K$.
\end{lemma}
\begin{proof}
	When $m$ is an invariant mean and $K\leq \Gamma$ is a subgroup with index $k<\infty$, we have
	\begin{equation}\label{eq:mFinIndex}
		m(K)=1/k,
	\end{equation}
	as $\Gamma$ is the disjoint union of $k$ translates of $K$.  For $C\subseteq \Gamma$, we therefore have $m(C+K)=N_{C}/k$, where $N_{C}$ is the number of cosets of $K$ having nonempty intersection with $K$.  This shows that $m(C)=|\tilde{C}|/k$ when $C$ is a union of cosets of $K$.  Since this holds for every invariant mean we have $\dens^{*}(C)=|\tilde{C}|/k$, as well.
\end{proof}

The main result of this section is Theorem \ref{th:MainOneMeanStrict}, regarding the inequality $m(A+B)<m(A)+m(B)$ for an arbitrary $m\in \mathcal M_{\tau}(\Gamma)$.  We first state a version with a stronger conclusion, under the additional hypothesis that $m$ is extreme.
\begin{theorem}\label{th:OneExtreme}
	Let $\Gamma$ be a discrete abelian group and $A$, $B\subseteq \Gamma$ nonempty.  There is a unique finite index subgroup $K\leq \Gamma$ (depending only on $A+B$) satisfying all of the following: if $\nu\in \mathcal M_{\tau}^{ext}(\Gamma)$ and  $\nu(A+B)<\nu(A)+\nu(B)$, then $\nu(A+B)=\nu(A+B+K)$, 
	\begin{equation}\label{eq:OneExtremeKneser}
		\nu(A+B+K)=\nu(A+K)+\nu(B+K)-\nu(K),
	\end{equation}
and $K$ is the stabilizer of $A+B+K$.  
\end{theorem}
 We call the subgroup $K$ the \emph{KJ-stabilizer} of $A+B$.  

Theorem \ref{th:OneExtreme} is a special case of Theorem \ref{th:TwoMeansStrict} below.

\begin{corollary}\label{cor:Q}
	Let $\Gamma$, $A$, $B$, and $K$ be as in Theorem \ref{th:OneExtreme}.  Then for every invariant mean $m$ on $\ell^{\infty}(\Gamma)$, we have
	\begin{equation}\label{eq:mGap}
		m(A)+m(B)-m(A+B)\leq \min\{m(A), m(B),[\Gamma:K]^{-1}\}
	\end{equation}
\end{corollary}
\begin{proof}  The bound $m(A)+m(B)-m(A+B)\leq \min\{m(A),m(B)\}$ follows immediately from the inequality $m(A+B)\geq \max\{m(A),m(B)\}$. 
	
	To prove $m(A)+m(B)-m(A+B)\leq [\Gamma:K]^{-1}$, let $g:=1_{A}+1_{B}-1_{A+B}$, so that $m(g)$ is the left-hand side of (\ref{eq:mGap}). 	Since $m\mapsto m(g)$ is a continuous linear functional on $\ell^{\infty}(\Gamma)^{*}$ and $\mathcal M_{\tau}(\Gamma)$ is a compact convex set, $\sup\{\nu(g):\nu\in \mathcal M_{\tau}(\Gamma)\}$ is attained by some $\nu\in \mathcal M_{\tau}^{ext}(\Gamma)$. Fix such a $\nu$ where the supremum is attained. To get $\nu(A)+\nu(B)-\nu(A+B)\leq \nu(K)$, note that $\nu(A)+\nu(B)-\nu(A+B)\leq \nu(A+K)+\nu(B+K)-\nu(A+B)$, and the latter expression simplifies to $\nu(K)$ by (\ref{eq:OneExtremeKneser}).  We have $\nu(K)=[\Gamma:K]^{-1}$ by Lemma \ref{lem:FiniteIndexMean}.
\end{proof}

We cannot obtain the conclusion of Theorem \ref{th:OneExtreme} under the weaker hypothesis $\nu\in \mathcal M_{\tau}(\Gamma)$:  applying Lemma \ref{lem:FolnerToMean} to the example shown in Remark \ref{rem:Example} produces an invariant mean $m$ on $\ell^{\infty}(\mathbb Z)$ and $A,B\subseteq \mathbb Z$ satisfying $m(A+B)=m(A)=m(B)=1/2$, while $m(A+B+K)=1$
for every finite index subgroup $K\leq \mathbb Z$.

In analogy with Theorem \ref{th:FolnerZ}, Theorem \ref{th:MainOneMeanStrict} recovers the desired conclusion by restricting to a subset of $\Gamma$ which is large and translation-invariant, with respect to $m$. A natural way to do this is to fix $C\subseteq \Gamma$ satisfying $m(C\triangle (C+\gamma))=0$ for all $\gamma\in \Gamma$, and  define a mean $m'$ by $m'(f):= m(C)^{-1}(f1_{C})$.  Since $m$ is not assumed to be countably additive, we require more flexibility.  See \S\ref{sec:Completeness} for discussion. 
\begin{definition}\label{def:StrongAbsoluteContinuity}
	Let $\Gamma$ be a discrete abelian group and let $m$ and $m'$ be (not necessarily invariant)  means on $\ell^{\infty}(\Gamma)$. For a given $\beta>0$, we say that $m'$ is \emph{a $\beta$-restriction of $m$}  if there is a sequence of sets $C_{n}\subseteq \Gamma$ satisfying 
	\begin{align}
		\label{eq:mCbeta}&\lim_{n\to\infty} m(C_{n})\geq \beta,\\
 		\label{eq:CnCauchy}&\lim_{N\to\infty} \sup_{n,n'\geq N} m(C_{n}\triangle C_{n'})=0,\\ 
 		\label{eq:Definem'}&m'(f)=\lim_{n\to\infty} m(C_n)^{-1}m(f1_{C_{n}})&& \text{for all } f\in \ell^{\infty}(\Gamma).
	\end{align}
\end{definition}
Note that if $m$ is invariant and (\ref{eq:mCbeta})-(\ref{eq:Definem'}) hold, then invariance of $m'$ is equivalent to
\begin{equation}
	\lim_{n\to\infty} m(C_{n}\triangle (C_{n}+\gamma))=0  \quad \text{for all } \gamma \in \Gamma.
\end{equation}

\begin{theorem}\label{th:MainOneMeanStrict}
	Let $\Gamma$ be a discrete abelian group with $A$, $B\subseteq \Gamma$ nonempty, and let $K\leq \Gamma$ be the KJ-stabilizer of $A+B$ given by Theorem \ref{th:OneExtreme}, so that $k:=[\Gamma:K]$ is finite and $K$ is the stabilizer of $A+B+K$.
	
	If $m\in \mathcal M_{\tau}(\Gamma)$ satisfies $\delta:=m(A)+m(B)-m(A+B)>0$, then there is an  $m'\in \mathcal M_{\tau}(\Gamma)$ such that
	\begin{enumerate}
		\item $m'$ is a $k\delta$-restriction of $m$;
		\item  $m'(A+B)=m'(A+B+K)$;
		\item $m'(A)+m'(B)-m'(A+B)\geq \delta.$
	\end{enumerate}
\end{theorem}
Theorem \ref{th:MainOneMeanStrict} is proved in \S\ref{sec:MainProofs}.  The proof combines Theorem \ref{th:OneExtreme} with Choquet decomposition applied to $\mathcal M_{\tau}(\Gamma)$.

\subsection{Obtaining the stabilizer equation}\label{sec:RecoverKEQ}

\begin{lemma}\label{lem:FItoKneser}
	Let $\Gamma$ be a discrete abelian group, $m\in \mathcal M_{\tau}(\Gamma)$, $K_{0}\leq \Gamma$ a finite index subgroup, and let $A$, $B\subseteq \Gamma$ be nonempty sets satisfying $m(A+B)<m(A)+m(B)$ and \begin{equation}\label{eq:StableK0}
		m(A+B)=m(A+B+K_{0}).
	\end{equation}  Let $K:=H(A+B+K_{0})$ be the stabilizer of $A+B+K_{0}$. Then $K_{0}\leq K\leq \Gamma$, $A+B+K_{0}=A+B+K$, and \begin{equation}\label{eq:KneserStable} m(A+B)=m(A+B+K)=m(A+K)+m(B+K)-m(K).
\end{equation}
\end{lemma}

\begin{proof}
Assuming (\ref{eq:StableK0}), let  $\phi:\Gamma\to \Gamma/K_{0}$ be the quotient map, and let $k_{0}$ be the index of $K_{0}$ in $\Gamma$.  For each $D\subseteq \Gamma$, let $\tilde{D}=\phi(D)$. We have $m(D+K_{0})=|\tilde{D}|/k_{0}$, by Lemma \ref{lem:FiniteIndexMean}. We see that \begin{align*}|\tilde{A}+\tilde{B}|=k_{0}m(A+B+K_{0})&=k_{0}m(A+B)\\
	&< k_{0}(m(A)+m(B))\\
	&\leq  k_{0}m(A+K_{0})+k_{0}m(B+K_{0})\\
	&=|\tilde{A}|+|\tilde{B}|.
\end{align*} Thus $|\tilde{A}+\tilde{B}|<|\tilde{A}|+|\tilde{B}|$. Let $H:=\{g\in \Gamma/K_{0}: \tilde{A}+\tilde{B}+g=\tilde{A}+\tilde{B}\}$.  Theorem \ref{th:SummenmengenSatz1} says that $H$ satisfies
\begin{equation}\label{eq:HKneser}
	|\tilde{A}+\tilde{B}|=|\tilde{A}+\tilde{B}+H|=|\tilde{A}+H|+|\tilde{B}+H|-|H|
\end{equation}  Let $K=\phi^{-1}(H)$, so that $K$ is a finite index subgroup of $\Gamma$ containing $K_{0}=\phi^{-1}(0_{\Gamma/K})$, with $m(K)=|H|(1/k_{0})$. Since $H$ is the stabilizer of $\tilde{A}+\tilde{B}$, we get that $K$ is the stabilizer of $A+B+K_{0}$. Equation (\ref{eq:KneserStable}) follows, since
\begin{align*} m(A+B)=m(A+B+K_{0})&=m(A+B+K)\\
	&=(1/k_{0})|\tilde{A}+\tilde{B}+H|\\
	&=(1/k_{0})(|\tilde{A}+H|+|\tilde{B}+H|-|H|)\\
	&=m(A+K)+m(B+K)-m(K). \qedhere
\end{align*}
\end{proof}

\subsection{Upper Banach density}\label{sec:UBD}
Theorem \ref{th:MainUBDstrict} will be derived from Theorem \ref{th:TwoMeansStrict}, an analogous statement with extreme invariant means in place of upper Banach density.  We write $C\sim_{\nu,\eta}D$ to mean $C\sim_{\nu} D$ and $C\sim_{\eta}D$.

\begin{theorem}\label{th:TwoMeansStrict}
	Let $\Gamma$ be a discrete abelian group and $A$, $B\subseteq \Gamma$ nonempty.  There is a unique finite index subgroup $K\leq \Gamma$ (depending only on $A+B$) satisfying all of the following: if $\nu$, $\eta\in \mathcal M_{\tau}^{ext}(\Gamma)$ satisfy
	\begin{equation}\label{eq:TwoExtremeHyp}
		\max\{\nu(A+B),\eta(A+B)\} <\nu(A)+\eta(B),
	\end{equation}
	we have $A+B\sim_{\nu,\eta} A+B+K$. Furthermore,  $K$ is the stabilizer of $A+B+K$, and
	\begin{equation}\label{eq:TwoKneser}
		\nu(A+B+K)=\nu(A+K)+\nu(B+K)-\nu(K).
	\end{equation}
\end{theorem}
We will deduce Theorem \ref{th:TwoMeansStrict} from Proposition \ref{prop:TwoMeansTechnical}.  We now prove Theorem \ref{th:MainUBDstrict}.

\begin{proof}[Proof of Theorem \ref{th:MainUBDstrict}]
	Assume $\Gamma$ is a discrete abelian group, and $A$, $B\subseteq \Gamma$ are nonempty sets satisfying $\dens^{*}(A+B)<\dens^{*}(A)+\dens^{*}(B)$. Let $K$ be the finite index subgroup given by Theorem \ref{th:TwoMeansStrict}. We will prove that $\dens^{*}(A+B)=\dens^{*}(A+B+K)$. Fix  $\nu$, $\eta\in \mathcal M_{\tau}^{ext}(\Gamma)$ with $\nu(A)=\dens^{*}(A)$ and $\eta(B)=\dens^{*}(B)$.  Then $\max\{\nu(A+B),\eta(A+B)\}\leq \dens^{*}(A+B)$ by definition of $\dens^{*}$, so (\ref{eq:TwoExtremeHyp}) holds.  Theorem \ref{th:TwoMeansStrict} says that $A+B\sim_{\nu} A+B+K$.  Thus $\nu(A+B)=\nu(A+B+K)$, and Lemma \ref{lem:FiniteIndexMean} implies $\nu(A+B+K)=\dens^{*}(A+B+K)$.  Now 
	\[\dens^{*}(A+B)\geq \nu(A+B)= \dens^{*}(A+B+K)\geq \dens^{*}(A+B),\] and we conclude that $\dens^{*}(A+B)=\dens^{*}(A+B+K)$.
\end{proof}

\subsection{Relation between \texorpdfstring{$m$}{dPhi} and \texorpdfstring{$\dens^{*}$}{d*}}

For many sets $C\subseteq \Gamma$ and invariant means $m$, we have $m(C)< \dens^{*}(C)$, so one might not expect the next corollary.
\begin{corollary}\label{cor:Unexpected}
	If $m$ is an invariant mean on $\ell^{\infty}(\Gamma)$ and $A$, $B\subseteq \Gamma$ satisfy $m(A+B)<m(A)+m(B)$, then $\dens^{*}(A+B)<\dens^{*}(A)+\dens^{*}(B)$. 
\end{corollary}

\begin{proof}
	Let $g=1_{A}+1_{B}-1_{A+B}$.  Note that $m(A+B)<m(A)+m(B) \iff m(g)>0$.  Since the map $m\mapsto m(g)$ is a continuous linear functional on $\ell^{\infty}(\Gamma)^{*}$, the supremum $\sup\{m(g):m\in \mathcal M_{\tau}(\Gamma)\}$ is attained at an extreme point $\nu$ of $\mathcal M_{\tau}(\Gamma)$.  Thus if $m(g)>0$ for some $m\in \mathcal M_{\tau}(\Gamma)$, then $\nu(g)>0$ for some $\nu\in \mathcal M_{\tau}^{ext}(\Gamma)$, meaning $\nu(A+B)<\nu(A)+\nu(B)$.  Now Theorem \ref{th:OneExtreme} implies $\dens^{*}(A+B)=\nu(A+B)<\nu(A)+\nu(B)\leq \dens^{*}(A)+\dens^{*}(B)$.
\end{proof}

\section{Extending the density theorems of Kneser, Bihani, and Jin}\label{sec:FolnerDef}

After defining F{\o}lner sequences and F{\o}lner nets for discrete abelian groups, we state Theorem \ref{th:MainFolnerStrict} below.  It generalizes Theorem \ref{th:FolnerZ} to discrete abelian groups.

\subsection{F{\o}lner nets} 

A \emph{directed set} $(I,\preceq)$ is a set $I$ together with a transitive relation $\preceq$ such that for which any two elements $i$, $j\in I$, there exists $\ell\in I$ such that $i\preceq \ell$ and $j\preceq \ell$.    A \emph{net} $(y_{i})_{i\in I}$ of elements of a set $Y$ is a function from a directed set $(I,\preceq)$ into $Y$. If $Y$ is a topological space and $y\in Y$, we write $\lim_{i\in I} y_{i} = y$ to mean that for every neighborhood $U$ of $y$, there is an $i_{U}\in I$ such that $y_{i}\in U$ whenever $i_{U}\preceq i$.

If $(y_{i})_{i\in I}$ is a net, a \emph{subnet}\footnote{We use the definition known as ``Willard subnet,'' see Chapter 7 of \cite{Schechter_Handbook} for detailed discussion of the various definitions.} $(y_{i(j)})_{j\in J}$ is given by a directed set $(J,\preceq')$ and a function $j\mapsto i(j)$ from $J$ to $I$ such that for all $j_{1}, j_{2}\in J$, $j_{1}\preceq' j_{2}$ $\implies i(j_{1})\preceq i(j_{2})$ and for all $i_{0}\in I$, there is a $j\in J$ such that $i_{0}\preceq i(j)$.  If $(y_{n})_{n\in \mathbb N}$ is a sequence, a subnet $(y_{n(j)})_{j\in J}$ is given by a directed set $(J,\preceq)$ and a function $j\mapsto n(j)$ from $J$ to $\mathbb N$ such that $i\preceq j$ $\implies$ $n(i)\leq n(j)$, and the set $\{n(j):j\in J\}$ is unbounded in $\mathbb N$.  If $Y$ is a compact topological space, then every net of elements of $Y$ has a convergent subnet; see \cite[p.~452]{Schechter_Handbook} for a proof.

\begin{definition} In a countable abelian group $\Gamma$, a \emph{F{\o}lner sequence} $\mb \Phi = (\Phi_{n})_{n\in \mathbb N}$ is a sequence of finite subsets of $\Gamma$ satisfying
	\begin{equation}\label{eq:Folner}
		\lim_{n\to\infty} \frac{|\Phi_n\triangle (\Phi_n+t)|}{|\Phi_n|}=0 \quad \text{ for every } t\in \Gamma.  
	\end{equation}For $A\subseteq \Gamma$, the \emph{upper density of $A$ along $\mathbf{\Phi}$} is $\udens_{\mb \Phi}(A):=\limsup_{n\to \infty} \frac{|A\cap \Phi_n|}{|\Phi_n|}$, while the \emph{lower density of $A$ along $\mathbf{\Phi}$} is $\ldens_{\mb \Phi}(A):=\liminf_{n\to \infty} \frac{|A\cap \Phi_n|}{|\Phi_n|}$.  When $\udens_{\mb \Phi}(A)=\ldens_{\mb \Phi}(A)$, we write $\dens_{\mb \Phi}(A)$ for the common value.
\end{definition}
When $\Gamma$ is uncountable no sequence of finite sets can satisfy (\ref{eq:Folner}); nets provide a satisfactory analogue. A \emph{F{\o}lner net} for $\Gamma$ is a net $(\Phi_{i})_{i\in I}$ of finite subsets of $\Gamma$ satisfying $\lim_{i\in I} \frac{|\Phi_i \triangle (\Phi_i+t)|}{|\Phi_i|}=0$ for every $t \in \Gamma$.  The \emph{lower density} and \emph{upper density} of $A$ \emph{along} $\mb \Phi$ are defined as $\ldens_{\mb \Phi}(A):= \lim_{i\in I} \inf_{i\preceq j} |A\cap \Phi_{j}|/|\Phi_{j}|$ and $\udens_{\mb \Phi}(A):= \lim_{i\in I} \sup_{i\preceq j} |A\cap \Phi_{j}|/|\Phi_{j}|$, respectively.
\begin{definition}
	A F{\o}lner net $\mb \Phi$ is \emph{full} if $\ldens_{\mb \Phi}(A)=\udens_{\mb \Phi}(A)$ for all $A\subseteq \Gamma$. Equivalently, $\mb \Phi$ is full if the function $m:\ell^{\infty}(\Gamma)\to \mathbb C$ given by $m(f):=\lim_{i\in I} \frac{1}{|\Phi_{i}|}\sum_{\gamma \in \Gamma} f(\gamma)$ is well defined, and consequently is an invariant mean.
\end{definition}
Weak$^{*}$ compactness of the unit ball in $\ell^{\infty}(\Gamma)^{*}$ implies that every F{\o}lner net has a subnet which is a full F{\o}lner net. This observation yields the following lemma.
\begin{lemma}\label{lem:FolnerToMean}
	Let $\mb \Phi$ be a F{\o}lner net for $\Gamma$.  Then there is an invariant mean $m$ on $\ell^{\infty}(\Gamma)$ such that $\ldens_{\mb \Phi}(A)\leq m(A) \leq \udens_{\mb \Phi}(A)$ for all $A\subseteq \Gamma$.
\end{lemma}
The next lemma is well-known; we give a new proof using a recent result.
\begin{lemma}\label{lem:UBD}
	Let $\Gamma$ be an abelian group and let $A\subseteq \Gamma$.  Then	
\begin{equation}\label{eq:UBDforNets}
	\dens^{*}(A)=	\sup \{\dens_{\mb \Phi}(A): \mb \Phi \text{ is a F{\o}lner net for } \Gamma\}.
\end{equation}
\end{lemma}

\begin{proof}
	Fix $A\subseteq \Gamma$ and let 
	\[\alpha := \sup\{m(A): m\in \mathcal M_{\tau}(\Gamma)\}, \quad \beta:=\sup\{\dens_{\mb \Phi}(A):\mb \Phi \text{ is a F{\o}lner net for } \Gamma\}.\] Lemma \ref{lem:FolnerToMean} implies  $
	\{\dens_{\mb \Phi}(A):\mb \Phi \text{ is a F{\o}lner net for } \Gamma\} \subseteq \{m(A):m\in \mathcal M_{\tau}(\Gamma)\}$,
	so $\beta \leq \alpha$.  To prove the reverse inequality, let $m$ be an invariant mean on $\ell^{\infty}(\Gamma)$.  Theorem 5.9 of \cite{Hopfensperger_WhenIsAMean} provides a F{\o}lner net $\mb \Phi$ such that $\dens_{\mb \Phi}(A)=m(A)$, so $\beta\geq \dens_{\mb \Phi}(A) \geq m(A)$.  Since this is true for every invariant mean $m$ on $\ell^{\infty}(\Gamma)$, we get $\beta\geq  \alpha$.
\end{proof}

\subsection{Kneser-type theorems for F{\o}lner sequences and nets}

The next proposition is an analogue of Theorem \ref{th:KneserLAD} for full F{\o}lner nets.  Theorem \ref{th:MainFolnerStrict} is the generalization to arbitrary F{\o}lner sequences and F{\o}lner nets.

\begin{proposition}\label{prop:FullFolnerStrict}
	Let $\Gamma$ be a discrete abelian group, $A$, $B\subseteq \Gamma$ nonempty, and let $K$ be the KJ-stabilizer of $A+B$ given by Theorem \ref{th:OneExtreme}, so that $K$ has finite index in $\Gamma$. If $\mb \Phi=(\Phi_{i})_{i\in I}$ is a \emph{full} F{\o}lner net satisfying 
	\[\delta:=\dens_{\mb \Phi}(A)+\dens_{\mb \Phi}(B)-\dens_{\mb \Phi}(A+B)>0,\]
	then there is a full F{\o}lner net $\mb \Psi=(\Psi_{i})_{i\in I}$ (with the same directed set $I$ as $\mb \Phi$)  such that $\Psi_{i}\subseteq \Phi_{i}$ for all $i\in I$, and each of the following hold:
	\begin{enumerate}
		\item$	\liminf_{i\in I} |\Psi_{i}|/|\Phi_{i}|\geq [\Gamma:K]\delta$;
\item\label{eq:GeneralKneser}
		$\dens_{\mb \Psi}(A+B)=\dens_{\mb \Psi}(A+B+K)$;
\item $\dens_{\mb \Psi}(A)+\dens_{\mb \Psi}(B) - \dens_{\mb \Psi}(A+B)\geq \delta$.
	\end{enumerate}
\end{proposition}
We prove Proposition \ref{prop:FullFolnerStrict} in \S\ref{sec:MainProofs}.  It will be derived from Theorem \ref{th:MainOneMeanStrict} above.

\begin{theorem}\label{th:MainFolnerStrict}
	Let $\Gamma$ be a discrete abelian group,  $A$, $B\subseteq \Gamma$ nonempty, and let $K$ be the KJ-stabilizer of $A+B$ given by Theorem \ref{th:OneExtreme}, so that $K$ has finite index in $\Gamma$. If $\mb \Phi=(\Phi_{i})_{i\in I}$ is a F{\o}lner net (F{\o}lner sequence) satisfying 
\begin{equation}\label{eq:DeltaHypothesis}\delta:=\ldens_{\mb \Phi}(A)+\ldens_{\mb \Phi}(B)-\ldens_{\mb \Phi}(A+B)>0,
\end{equation}
then there is a subnet $\mb \Phi'=(\Phi_{i(j)})_{j\in J}$ (subsequence $(\Phi_{n_{j}})_{j\in \mathbb N}$) and a F{\o}lner net (F{\o}lner sequence) $\mb \Psi$ such that $\Psi_{j}\subseteq \Phi_{i(j)}$ for all $j\in J$, and each of the following hold:
	\begin{enumerate}
	\item\label{eq:subsetDense} $	\liminf_{j\in J} |\Psi_{j}|/|\Phi_{i(j)}|\geq [\Gamma:K]\delta$;
	\item\label{eq:SubnetGeneralKneser}
	$\dens_{\mb \Psi}(A+B)=\dens_{\mb \Psi}(A+B+K)$;
	\item\label{eq:SubnetDelta} $\dens_{\mb \Psi}(A)+\dens_{\mb \Psi}(B) - \dens_{\mb \Psi}(A+B)\geq \delta$.
\end{enumerate}
\end{theorem}

\begin{remark}
	The hypothesis (\ref{eq:DeltaHypothesis}) may be weakened to
	\begin{equation}\label{eq:WeakerHyp}
		\delta:=\limsup_{i\in I} |A\cap \Phi_{i}|/|\Phi_{i}|+|B\cap \Phi_{i}|/|\Phi_{i}| - |(A+B)\cap \Phi_{i}|/|\Phi_{i}|>0
	\end{equation}
without altering the conclusion: assuming (\ref{eq:WeakerHyp}), some subnet of $\mb \Phi$  satisfies (\ref{eq:DeltaHypothesis}).  
\end{remark}

\begin{proof}[Proof of Theorem \ref{th:MainFolnerStrict}] When $\mb \Phi$ is a F{\o}lner net, Theorem \ref{th:MainFolnerStrict} follows immediately from Proposition \ref{prop:FullFolnerStrict}: every subnet $\mb \Phi'$ of $\mb \Phi$ with $\dens_{\mb \Phi'}(A+B)=\ldens_{\mb \Phi}(A+B)$ has a subnet $\mb \Phi''$ which is full and satisfies $\dens_{\mb \Phi''}(A)+\dens_{\mb \Phi''}(B)-\dens_{\mb \Phi''}(A+B)\geq \delta$.

When $\mb \Phi$ is a F{\o}lner sequence, $\Gamma$ is necessarily countable. We find the desired sequence $\mb \Psi$ as follows: let $\mb \Phi'=(\Phi_{n(j)})_{j\in J}$ be a subnet of $\mb \Phi$ which is a full F{\o}lner net, and let $\mb \Psi=(\Psi_{n(j)})_{j\in J}$ be the corresponding F{\o}lner net satisfying the conclusion of Proposition \ref{prop:FullFolnerStrict}.

Below we will write $\dens_{F}(C)$ for $|C\cap F|/|F|$ when $F\subseteq \Gamma$ is finite. 

 Write $\Gamma$ as a countable union of finite subsets $E_{m}$, $m\in \mathbb N$.  From the conclusion of Proposition (\ref{prop:FullFolnerStrict}), we may find a sequence of indices $j(m)\in J$ satisfying each of the following:
\begin{equation}\label{eq:Index}
n(j(1))<n(j(2))<\dots
\end{equation}
\begin{equation}
\label{eq:PsiDense} |\Psi_{n(j(m))}|/|\Phi_{n(j(m))}|\geq [\Gamma:K]\delta-\frac{1}{m};
\end{equation}
\begin{equation}\label{eq:PsiInFolner}	\frac{|\Psi_{n(i(m))}\triangle (\Psi_{n(j(m))}+t)|}{|\Psi_{n(j(m))}|}< \frac{1}{m} \quad \text{for all } t\in E_{m};
\end{equation}
\begin{equation}\label{eq:PsiAlmostDStar}
	\dens_{\Psi_{n(j(m))}}(A+B) > \dens^{*}(A+B)-\frac{1}{m};
\end{equation}
\begin{equation}\label{eq:PsiDelta}	\dens_{\Psi_{n(j(m))}}(A)+\dens_{\Psi_{n(j(m))}}(B)-\dens_{\Psi_{n(j(m))}}(A+B)> \delta -\frac{1}{m}.
\end{equation}
Let $\mb \Psi = (\Psi_{n(j(m))})_{m\in \mathbb N}$.  Then $\mb \Psi$ is a F{\o}lner sequence, by (\ref{eq:PsiInFolner}).    By (\ref{eq:Index}), $(\Phi_{n(j(m))})_{m\in \mathbb N}$ is a subsequence of $\mb \Phi$.  Conclusions \ref{eq:subsetDense}, \ref{eq:SubnetGeneralKneser}, and \ref{eq:SubnetDelta} follow from \ref{eq:PsiDense}, \ref{eq:PsiAlmostDStar}, and \ref{eq:PsiDelta}.
\end{proof}

\section{Compact models for discrete sumsets}\label{sec:Models}

\subsection{Characters}  We write $\mathcal S^{1}$ for the circle group $\{z\in \mathbb C: |z|=1\}$, with the group operation of multiplication and the usual topology.

Let $G$ be an LCA group.  A \emph{character} of $G$ is a continuous homomorphism $\chi:G\to \mathcal S^{1}$.  The set of all characters of $G$ is denoted $\widehat{G}$.  Note that when $G$ is discrete, every homomorphism $\chi:G\to \mathcal S^{1}$ is continuous.

\subsection{Bohr compactification} 
Let $\Gamma$ be a discrete abelian group.  The \emph{Bohr compactification}  of $\Gamma$ is a compact abelian group $b\Gamma$, together with a one-to-one homomorphism $\iota:\Gamma\to b\Gamma$, such that
\begin{enumerate}
	\item[(i)] $\iota(\Gamma)$ is a topologically dense subgroup of $b\Gamma$;
	\item[(ii)]\label{item:bGamma2} for every character $\chi\in \widehat{\Gamma}$, there is a continuous character $\tilde{\chi}\in \widehat{b\Gamma}$ such that $\chi = \tilde{\chi}\circ \iota$.
\end{enumerate}
We will identify $\Gamma$ with its image $\iota(\Gamma)$ in $b\Gamma$, meaning we consider $\Gamma$ as a topologically dense subgroup of $b\Gamma$.  With this identification, (ii) says that every $\chi\in \widehat{\Gamma}$ extends uniquely to a continuous character $\tilde{\chi}\in \widehat{b\Gamma}$.

We will write $\mu_{b\Gamma}$ for Haar probability measure on $b\Gamma$.

The following is Theorem 1.5 of \cite{Griesmer_DiscreteSumsets}. A subset of a topological space is an \emph{$F_{\sigma}$-set} if it is a countable union of compact sets.  Note that if $\tilde{A}$, $\tilde{B}\subseteq b\Gamma$ are $F_{\sigma}$-sets, then so is $\tilde{A}+\tilde{B}$, and consequently $\tilde{A}+\tilde{B}$ is $\mu_{b\Gamma}$-measurable.

\begin{theorem}\label{th:Model}  
	Let $\Gamma$ be a discrete abelian group.  Let $\nu\in \mathcal M_{\tau}^{ext}(\Gamma)$, $m\in \mathcal M_{\tau}(\Gamma)$, and $A$, $B\subseteq \Gamma$.  
	There are $F_{\sigma}$ sets $\tilde{A}_{\nu}$, $\tilde{B}_{m}\subseteq b\Gamma$  such that $\mu_{b\Gamma}(\tilde{A}_{\nu})\geq \nu(A)$, $\mu_{b\Gamma}(\tilde{B}_{m})\geq m(B)$, and
	\begin{equation}
		\nu(A+B)\geq \mu_{b\Gamma}(\tilde{A}_{\nu}+\tilde{B}_{m}).
	\end{equation}
	Furthermore,
	\begin{enumerate}
		\item\label{item:AVBV} if $V\subseteq b\Gamma$ is compact, then 
		\[\mu_{b\Gamma}(\tilde{A}_{\nu}\cap V)\geq \nu(A\cap V) \quad \text{and} \quad \mu_{b\Gamma}(\tilde{B}_{m}\cap V)\geq m(B\cap V);\]
		
		\item\label{item:VbGammaClopen} if $V\subseteq b\Gamma$ is clopen, then
		\[\mu_{b\Gamma}(\tilde{A}_{\nu}\cap V)=0\iff \nu(A\cap V)=0 \quad \text{and} \quad \mu_{b\Gamma}(\tilde{B}_{m}\cap V)=0\iff m(B\cap V)=0;\]
		
		\item\label{item:VcapGamma}  if $V\subseteq \tilde{A}_{\nu}+\tilde{B}_{m}$ is compact, then $V\cap \Gamma \subset_{\nu} A+B$.
	\end{enumerate}
\end{theorem}

The sets $\tilde{A}_{\nu}$, $\tilde{B}_{m}$, and $\tilde{A}_{\nu}+\tilde{B}_{m}\subseteq b\Gamma$ in Theorem \ref{th:Model} are models of $A$, $B$, and $A+B$.

\subsection{Finite index subgroups}  Fix a discrete abelian group $\Gamma$.

\begin{lemma}\label{lem:FiniteIndexCorrespondence}
  Let $\tilde{K}\leq  b\Gamma$ be a $\mu_{b\Gamma}$-measurable finite index subgroup and let $K=\tilde{K}\cap \Gamma$. Then
	
	\begin{enumerate}
		\item[(i)]  $K$ has finite index in $\Gamma$, and $[\Gamma:K]=[b\Gamma:\tilde{K}]$.
		
		\item[(ii)] If $x+\tilde{K}$ is a coset of $\tilde{K}$, then $(x+\tilde{K})\cap \Gamma$ is a coset of $K$.
		
		\item[(iii)] If $C\subseteq \Gamma$, then $C+K = (C+\tilde{K})\cap \Gamma$.
		
		\item[(iv)] The map $\rho: \Gamma/K\to b\Gamma/\tilde{K}$ given by $\rho(\gamma+K)=\gamma+\tilde{K}$ is a group isomorphism; its inverse is given by $\rho^{-1}(x+\tilde{K})=(x+\tilde{K})\cap \Gamma$.
		
		\item[(v)] If $C\subseteq b\Gamma$ and $\eta$ is an invariant mean on $\ell^{\infty}(\Gamma)$, then $\eta((C+\tilde{K})\cap \Gamma)=\mu_{b\Gamma}(C+\tilde{K})$.
	\end{enumerate}
\end{lemma}

See \cite[Lemma 11.1]{Griesmer_DiscreteSumsets} for a proof. In the sequel we use the well-known\footnote{See the proof of Lemma 11.1 (i) in \cite{Griesmer_DiscreteSumsets} for a short explanation.} fact that a subgroup of a compact abelian group has finite index if and only if it is clopen.

\begin{corollary}\label{cor:MassInFiniteIndex}
	Let $m\in \mathcal M_{\tau}(\Gamma)$, let $\tilde{K}\leq b\Gamma$ be a $\mu_{b\Gamma}$-measurable finite index subgroup, and let $K=\tilde{K}\cap b\Gamma$.   Let $\tilde{K}_{i}$,  $i=1,\dots,r$ enumerate the cosets of $\tilde{K}$, so that $K_{i}:=\tilde{K}_{i}\cap \Gamma$ enumerates the cosets of $K$.  Let $A\subseteq \Gamma$ with $m(A)>0$ and let $\gamma\in \Gamma$.  Then:
	\begin{enumerate} 
		\item[(i)]  $\mu_{b\Gamma}(\tilde{A}_{m}\cap \tilde{K}_{i})>0$ if and only if $m(A\cap K_{i})>0$;
		\item[(ii)] setting $A':=\bigcup \{A\cap K_i : i\leq r,\, m(A\cap K_i)>0\}$, we have		
		\begin{align} 
			\label{eq:nuA'nuA}	m(A')&=m(A) \text{ and } A'\sim_{m} A;\\
			\label{eq:A'KA'tildeK} A' + K &= (A'+\tilde{K})\cap \Gamma;\\
			\label{eq:nuA'KmuA'}		 m(A'+K) &=\mu_{b\Gamma}(A'+\tilde{K}).
		\end{align}
	\end{enumerate}
\end{corollary}

\begin{proof}
	(i)	The inequality $\mu_{b\Gamma}(\tilde{A}_m \cap \tilde{K}_i)\geq m(A\cap K_i)$ follows from Theorem \ref{th:Model}, since $A\cap \tilde{K}_i=A\cap K_i$.  Thus it suffices to prove that if $m(A\cap K_i)=0$, then $\mu_{b\Gamma}(\tilde{A}_m \cap K_i)=0$.  Since $\tilde{K}$ is clopen in $b\Gamma$, this follows from Theorem \ref{th:Model} \ref{item:VbGammaClopen}.

	 (\ref{eq:nuA'nuA}) follows from the definition of $A'$ and the fact that $K$ has finite index in $\Gamma$.
	
	 (\ref{eq:A'KA'tildeK}) and (\ref{eq:nuA'KmuA'}) follow from Lemma \ref{lem:FiniteIndexCorrespondence}.
\end{proof}

\section{Proof of Theorem \ref{th:TwoMeansStrict}}\label{sec:UBDproof}

In the next subsection we prove some technical lemmas.  In \S\ref{sec:TwoExtreme} we prove Theorem \ref{th:TwoMeansStrict}.

\subsection{Consequences of Theorem \ref{th:SummenmengenSatz1}}

Recall Theorem \ref{th:SummenmengenSatz1}: if $G$ is an LCA group with Haar measure $\mu$ and $A, B\subseteq G$ are nonempty $\mu$-measurable sets satisfying $\mu_{*}(A+B)<\mu(A)+\mu(B)$, then $K:=H(A+B)$ is compact, open, and satisfies
\[\mu(A+B)=\mu(A+K)+\mu(B+K)-\mu(K).\]

\begin{lemma}\label{lem:Push}
	Let $G$ be a compact abelian group with Haar measure $\mu$, and let $A$, $B\subseteq G$ satisfy $\mu_{*}(A+B)<\mu(A)+\mu(B)$.  Let $H:=H(A+B)$ and let $g\in G$.  Then
	\begin{enumerate}
\item[(a)]  $\mu(A\cap (g+H))>0$ if and only if $A\cap (g+H)\neq \varnothing$.  Likewise $\mu(B\cap (g+H))>0$ if and only if $B\cap (g+H)\neq \varnothing$.

\item[(b)]	If $A+g\subseteq A+B$, then $g \in B+H$.  Likewise if $g + B\subseteq A+B$, then $g \in A+H$.	
\end{enumerate}
\end{lemma}

\begin{proof}
(a) See \cite[Lemma 9.6]{Griesmer_Semicontinuity} for a proof, which  is very similar to the proof of Lemma \ref{lem:Jin} below.

(b) Assume, to get a contradiction, that the hypotheses hold, $A+g\subseteq A+B$, and $g\notin B+H$.  Let $B':=B\cup\{g\}$.  Then 
	\begin{equation}\label{eq:B'H}
	\mu(B'+H)=\mu(B+H)+\mu(H),
	\end{equation}
	since $g+H\subseteq B'+H$ and $g+H$ is disjoint from $B+H$.
	
	Since $A+B\subseteq A+B'$ and $A+B'\subseteq A+B$, we have $A+B=A+B'$.  It follows that $H(A+B')=H(A+B)=H$.  	We also have $\mu(B')\geq \mu(B)$, so \begin{equation}\label{eq:muAB'}
		\mu(A+B')< \mu(A)+\mu(B').
	\end{equation}  Theorem \ref{th:SummenmengenSatz1} then implies
	\begin{align*}
		\mu(A+B')&=\mu(A+H)+\mu(B'+H)-\mu(H)\\
		&=\mu(A+H)+\mu(B+H)+\mu(H)-\mu(H) && \text{by (\ref{eq:B'H})}\\
		&=\mu(A+H)+\mu(B'+H)\\
		&\geq \mu(A)+\mu(B).
	\end{align*} Combining the above with (\ref{eq:muAB'}), we get $\mu(A+B)=\mu(A+B')\geq \mu(A)+\mu(B)$, contradicting the hypothesis that $\mu_{*}(A+B)<\mu(A)+\mu(B)$.
\end{proof}

The following is essentially Lemma 2.5 in \cite{Jin_CharacterizingStructure}.  Note that we do not assume $m$ is extreme.

\begin{lemma}\label{lem:Jin}  Let $\Gamma$ be a discrete abelian group.  Suppose $\nu\in \mathcal M_{\tau}^{ext}(\Gamma)$, $m\in \mathcal M_{\tau}(\Gamma)$, and $A$, $B\subseteq \Gamma$ are nonempty.  Let $\tilde{A}_{\nu}$, $\tilde{B}_{m}\subseteq b\Gamma$ be as in Theorem \ref{th:Model}, and assume  \[\mu_{b\Gamma}(\tilde{A}_\nu+\tilde{B}_m)<\mu_{b\Gamma}(\tilde{A}_\nu)+\mu_{b\Gamma}(\tilde{B}_m),\]  so that $\tilde{H}:=H(\tilde{A}_\nu+\tilde{B}_m)$ is $\mu_{b\Gamma}$-measurable and finite index in $b\Gamma$ by Theorem \ref{th:Satz1Compact}. Let $\tilde{H}_1,\dots,\tilde{H}_d$ enumerate the cosets of $\tilde{H}$, and fix some $j\leq d$ with $\nu(A\cap \tilde{H}_j)>0$.  Then
	\begin{equation}\label{eq:JinLemma}\nu(A\cap \tilde{H}_j)+\mu_{b\Gamma}(\tilde{A}_\nu+\tilde{B}_m)\geq \nu(A)+m(B).
	\end{equation}
\end{lemma}

\begin{proof} Write $H$ for $\tilde{H}\cap \Gamma$ and $A_i$ for $A\cap \tilde{H}_i$. For the remainder of the proof we write $\mu$, $\tilde{A}$, and $\tilde{B}$ for $\mu_{b\Gamma}$, $\tilde{A}_{\nu}$, and $\tilde{B}_{m}$, respectively.  Our assumption and Theorem \ref{th:SummenmengenSatz1} imply
	\begin{equation}\label{eq:KneserForTransfer}
		\mu(\tilde{A}+\tilde{B})=\mu(\tilde{A}+\tilde{H})+\mu(\tilde{B}+\tilde{H})-\mu(\tilde{H}).
	\end{equation} 
	Let $I:=\{i:\nu(A\cap \tilde{H}_i)>0\}$, and let $A':=\bigcup_{i\in I} A_i$.  Then $\nu(A'+H)=|I|\nu(H)$, and
	\begin{align*}\nu(A)=\nu(A')&=\sum_{i\in I} \nu(A_i)\\
		&=\nu(A_j)+\sum_{i\in I, i\neq j} \nu(A_i)\\
		&\leq \nu(A_j) + \sum_{i\in I, i\neq j} \nu(H)\\
		&= \nu(A_j) + \nu(A'+H)-\nu(H)
	\end{align*}
	Subtracting $\nu(A'+H)-\nu(H)$ from the first and last lines above we get
	\begin{equation}\label{eq:SubtractFirstLast}
		\nu(A_j)\geq \nu(A) -(\nu(A'+H)-\nu(H)).
	\end{equation}
	Since $\nu$ and $m$ are invariant means, Lemmas \ref{lem:FiniteIndexCorrespondence} and \ref{lem:FiniteIndexMean} imply $\nu(H)=\mu(\tilde{H})$, while Theorem \ref{th:Model} says that $\mu(\tilde{A})\geq \nu(A)$ and $\mu(\tilde{B})\geq m(B)$.
	
	Adding $\mu(\tilde{A}+\tilde{B})$ to both sides of (\ref{eq:SubtractFirstLast}), we have
	\begin{align*} \nu(A_j)&+\mu(\tilde{A}+\tilde{B}) \\ &\geq \nu(A) -(\nu(A'+H)-\nu(H)) + \mu(\tilde{A}+\tilde{B})\\
		&= \nu(A) - (\nu(A'+H)-\nu(H)) + \mu(\tilde{A}+\tilde{H})+\mu(\tilde{B}+\tilde{H})-\mu(\tilde{H})&& \text{by (\ref{eq:KneserForTransfer})}\\
		&= \nu(A) + \mu(\tilde{B}+\tilde{H}) +  \mu(\tilde{A}+\tilde{H})  -\nu(A'+H) +\nu(H) -\mu(\tilde{H}) \\
		&= \nu(A) + \mu(\tilde{B}+\tilde{H}) && \text{by Corollary \ref{cor:MassInFiniteIndex}}\\
		&\geq \nu(A) + \mu(\tilde{B})\\
		&\geq \nu(A)+m(B),
	\end{align*}
	as desired.  \end{proof}

Most of the proof of Theorem \ref{th:TwoMeansStrict} is contained in the next lemma.  Note that we only assume $\nu$ is extreme here.
\begin{lemma}\label{lem:HalfCritical}
	Let $\Gamma$ be a discrete abelian group, $\nu\in \mathcal M_{\tau}^{ext}(\Gamma)$, $m\in \mathcal M_{\tau}(\Gamma)$, and $A$, $B\subseteq \Gamma$.  Assume  $\nu(A+B)<\nu(A)+m(B)$.
	Then
	\begin{equation}\label{eq:ModelInequality}
		\mu_{b\Gamma}(\tilde{A}_{\nu}+\tilde{B}_{m})<\mu_{b\Gamma}(\tilde{A}_{\nu})+\mu_{b\Gamma}(\tilde{B}_{m}),
	\end{equation}
	$\tilde{H}:=H(\tilde{A}_{\nu}+\tilde{B}_{m})$ has finite index in $b\Gamma$ and satisfies 
	\begin{align}\label{eq:nuContain}
		(\tilde{A}_{\nu}+\tilde{B}_{m}+\tilde{H})\cap \Gamma &\subset_{\nu} A+B, \quad \text{and}\\
	 \label{eq:BtildeB} B&\subseteq \tilde{B}_{m}+\tilde{H}.
\end{align}
\end{lemma}
\begin{proof}  Assuming $\nu(A+B)<\nu(A)+m(B)$, Theorem \ref{th:Model} implies \[\mu_{b\Gamma}(\tilde{A}_{\nu}+\tilde{B}_{m})\leq \nu(A+B)<\nu(A)+m(B)\leq \mu_{b\Gamma}(\tilde{A}_{\nu})+\mu_{b\Gamma}(\tilde{B}_{m}).\]
	This confirms (\ref{eq:ModelInequality}).  Theorem \ref{th:Satz1Compact} now implies $\tilde{H}$ is compact, open, has finite index in $b\Gamma$, and satisfies $\tilde{A}_{\nu}+\tilde{B}_{m}=\tilde{A}_{\nu}+\tilde{B}_{m}+\tilde{H}$.  Thus $\tilde{A}_{\nu}+\tilde{B}_{m}$ is clopen, so Theorem \ref{th:Model} \ref{item:VcapGamma} implies (\ref{eq:nuContain}).
	
	Lemma \ref{lem:Push} (a) and Corollary \ref{cor:MassInFiniteIndex}  (i) then imply
	\begin{equation}\label{eq:TooManyIFFs}
		\tilde{A}_{\nu}\cap (\gamma+\tilde{H})\neq \varnothing \iff \mu_{b\Gamma}(\tilde{A}_{\nu}\cap (\gamma+\tilde{H}))>0 \iff \nu(A\cap (\gamma+\tilde{H}))>0.
	\end{equation}
	
	To prove (\ref{eq:BtildeB}), we assume, to get a contradiction, that $b_{0}\in B$ and $b_{0}\notin \tilde{B}_{m}+\tilde{H}$.  Applying the contrapositive of Lemma \ref{lem:Push} (b) with $\tilde{A}_{\nu}$, $\tilde{B}_{m}$, and $\tilde{H}$ in place of $A$, $B$, and $H$, we see that $\tilde{A}_{\nu}+b_{0}$ is not contained in $\tilde{A}_{\nu}+\tilde{B}_{m}+\tilde{H}$.  So we may fix $a_{0}\in A$ so that $\tilde{A}_{\nu}\cap (a_{0}+\tilde H)$ is nonempty, while
		\begin{equation}
		C_{0}:=(\tilde{A}_{\nu}\cap (a_{0}+\tilde H)) + b_{0}
	\end{equation}
is disjoint from $\tilde{A}_{\nu}+\tilde{B}_{m}+\tilde{H}$.  Then (\ref{eq:TooManyIFFs}) implies $\nu(A\cap (a_{0}+\tilde{H}))>0$.   Lemma \ref{lem:Jin} now implies 
		 \begin{equation}\label{eq:nuAcapa0} \nu(C_{0})+\mu_{b\Gamma}(\tilde{A}_{\nu}+\tilde{B}_{m})\geq \nu(A)+m(B).
		 \end{equation}  Together with (\ref{eq:nuContain}) and the disjointness of $C_{0}$ from $\tilde{A}_{\nu}+\tilde{B}_{m}+\tilde{H}$, (\ref{eq:nuAcapa0}) implies $\nu(A+B)\geq \nu(A)+m(B)$, contradicting the hypothesis $\nu(A+B)<\nu(A)+m(B)$.
\end{proof}

\subsection{Two extreme invariant means}\label{sec:TwoExtreme}

\begin{proposition}\label{prop:TwoMeansTechnical}
	Let $\Gamma$ be a discrete abelian group, $\nu$, $\eta\in \mathcal M_{\tau}^{ext}(\Gamma)$ and $A$, $B\subseteq \Gamma$. Let $\tilde{A}_{\nu}$, $\tilde{B}_{m}\subseteq b\Gamma$ be as in Theorem \ref{th:Model}.  Assume \[\max\{\nu(A+B),\eta(A+B)\} <\nu(A)+\eta(B),\]  define
	$\tilde{C}:=\tilde{A}_\nu+\tilde{B}_\eta$, $C:=\tilde{C}\cap \Gamma$,
	$\tilde{H}:=H(\tilde{C})$
	and let $H:=\tilde{H}\cap \Gamma$.  Then $H$ has finite index in $\Gamma$, and the following hold:
	\begin{enumerate}
		\item\label{item:CisAplusB} $C=A+B+H$.
		\item\label{item:StrictSim} $A+B\sim_{\nu,\eta} A+B+H$.	
		\item\label{item:HisH} $H$ is the stabilizer of $A+B+H$.	
		\item\label{item:Kneser} $\nu(A+B+H)=\nu(A+H)+\nu(B+H)-\nu(H)$.
	\end{enumerate}
\end{proposition}

\begin{proof}  Assume $A$, $B$, $\nu$, and $\eta$ are as in the hypothesis.  Since both $\nu, \eta\in \mathcal M_{\tau}^{ext}(\Gamma)$, we may apply Lemma \ref{lem:HalfCritical} to conclude that $B\subseteq \tilde{B}_{\eta}+\tilde{H}$.  Applying Lemma \ref{lem:HalfCritical} with the roles of $\nu$ and $\eta$ reversed, we get $A\subseteq \tilde{A}_{\nu}+\tilde{H}$.  Thus $A+B\subseteq \tilde{A}_{\nu}+\tilde{B}_{\eta}+\tilde{H}=\tilde{A}_{\nu}+\tilde{B}_{\eta}$, meaning $A+B\subseteq \tilde{C}$.  Since $A+B\subseteq \Gamma$, we have $A+B\subseteq C:=\tilde{C}\cap \Gamma$.
	
	The group $\tilde{H}$ has finite index in $b\Gamma$ by Lemma \ref{lem:HalfCritical}, so $H$ has finite index in $\Gamma$ by Lemma \ref{lem:FiniteIndexCorrespondence}.  
	
	We claim that $C+H=C$.  To see this, note that Lemma \ref{lem:FiniteIndexCorrespondence} implies $C+H=(C+\tilde{H})\cap \Gamma \subseteq (\tilde{C}+\tilde{H})\cap \Gamma = C$. Now Lemma \ref{lem:HalfCritical} and the above containments imply
\begin{equation}
	C:=(\tilde{A}_{\nu}+\tilde{B}_{\eta}+\tilde{H})\cap \Gamma\subset_{\nu,\eta} A+B\subseteq C.
\end{equation}
Thus $A+B\sim_{\nu,\eta} C$.  Furthermore $A+B+H\subseteq C+H=C$, so $A+B\sim_{\nu,\eta}A+B+H$. Conclusion \ref{item:CisAplusB} now follows from the containment $A+B+H\subseteq C$ and the similarity $A+B\sim_{\nu}A+B+H$.

Now we prove that $H=H(C)$.  The containment $H\subseteq H(C)$ follows from $C+H=C$ above, so we prove the reverse containment.   Assume $\gamma\in H(C)$.   Then 
\[\tilde{C}\cap \Gamma =C=C+\gamma=(\tilde{C}\cap \Gamma)+\gamma=(\tilde{C}+\gamma)\cap \Gamma.\]  Thus $\tilde{C}\cap \Gamma=(\tilde{C}+\gamma)\cap \Gamma$.  Since $\tilde{C}$ is a union of cosets of $\tilde{H}$, $\tilde{H}$ is clopen in $b\Gamma$, and $\Gamma$ is dense in $b\Gamma$, the latter equation implies $\tilde{C}=\tilde{C}+\gamma$.  It follows that $\gamma\in \tilde{H}$ and $\gamma\in \Gamma$, so $\gamma\in H$. We conclude that $H=H(C)$.  Together with \ref{item:CisAplusB}, this implies \ref{item:HisH}.  

To prove \ref{item:Kneser}, we will apply Lemma \ref{lem:FItoKneser}, with $\nu$ in place of $m$ and $A+H$ and $B+H$ in place of $A$ and $B$.  We first verify that $\nu(A+B+H)<\nu(A+H)+\nu(B+H)$.  To see this, note that Lemma \ref{lem:FiniteIndexMean} implies $\nu(A+H)=\eta(A+H)$, so \[\nu(A+B+H)=\nu(A+B)<\nu(A)+\eta(B)\leq \nu(A+H)+\eta(B+H)=\nu(A+H)+\nu(B+H).\]  Since $H$ is the stabilizer of $A+B+H$, Lemma \ref{lem:FItoKneser} yields \ref{item:Kneser}. \end{proof}

The next lemma will show that the subgroup $H$ in Proposition \ref{prop:TwoMeansTechnical} does not depend on the choice of $\nu$ and $\eta$.

\begin{lemma}\label{lem:TwoSubgroups}
Let $\Gamma$ be a discrete abelian group and $C\subseteq \Gamma$.	If $K$, $H\leq \Gamma$ are each finite index subgroups, $K$ is stabilizer of $C+K$, $H$ is the stabilizer of $C+H$, and $\dens^{*}(C)=\dens^{*}(C+H)=\dens^{*}(C+K)$, then $H=K$.
\end{lemma}
\begin{proof}
Let $L:=H\cap K$, so that $L$ has finite index in $\Gamma$. Since $L\subseteq H$, we have $\dens^{*}(C)\leq \dens^{*}(C+L)\leq \dens^{*}(C+H)$.  Since $\dens^{*}(C)=\dens^{*}(C+H)$, this implies  $\dens^{*}(C+L)=\dens^{*}(C+H).$ Since $L$ and $H$ both have finite index, Lemma \ref{lem:FiniteIndexMean} then implies $C+L=C+H$.  Likewise $C+L=C+K$, so $C+H=C+K$.  Then $C+H+K=C+K+K=C+K=C+H$, so $K$ is contained in the stabilizer of $C+H$, meaning $K\subseteq H$.  By symmetry, we get $H\subseteq K$, so $H=K$.
\end{proof}

We now prove Theorem \ref{th:TwoMeansStrict}.  Recall the statement: 	Let $\Gamma$ be a discrete abelian group and $A$, $B\subseteq \Gamma$ nonempty.  There is a unique finite index subgroup $K\leq \Gamma$ (depending only on $A+B$) satisfying all of the following:  if $\nu$, $\eta\in \mathcal M_{\tau}^{ext}(\Gamma)$ satisfy
\begin{equation}\label{eq:TwoExtremeHypRepeat}
	\max\{\nu(A+B),\eta(A+B)\} <\nu(A)+\eta(B),
\end{equation}
we have $A+B\sim_{\nu,\eta} A+B+K$. Furthermore,  $K$ is the stabilizer of $A+B+K$, and 
\begin{equation}\label{eq:TwoKneserRepeat}
	\nu(A+B+K)=\nu(A+K)+\nu(B+K)-\nu(K).
\end{equation}

\begin{proof}[Proof of Theorem \ref{th:TwoMeansStrict}]
Assuming $\nu, \eta\in \mathcal M_{\tau}^{ext}(\Gamma)$ satisfy (\ref{eq:TwoExtremeHypRepeat}), we may apply Proposition \ref{prop:TwoMeansTechnical}.  With the subgroup $H$ in the conclusion of Proposition \ref{prop:TwoMeansTechnical}, we set $K=H$, so that $K$ has finite index, $K=H(A+B+K)$, and $A+B\sim_{\nu,\eta}A+B+K $.  By Lemma  \ref{lem:TwoSubgroups}, the same conclusion holds, with the same subgroup $K$, for any two extreme invariant means $\nu$, $\eta$ satisfying (\ref{eq:TwoExtremeHypRepeat}).  By Proposition \ref{prop:TwoMeansTechnical} \ref{item:HisH} and \ref{item:Kneser}, we get that $K$ is the stabilizer of $A+B+K$, and equation (\ref{eq:TwoKneserRepeat}) holds.

To prove the asserted uniqueness of $K$, assume that $K_{1}$, $K_{2}\leq \Gamma$ have finite index and satisfy the conclusion of Theorem \ref{th:TwoMeansStrict}.  Fix $\nu, \eta\in \mathcal M_{\tau}^{ext}(\Gamma)$ satisfying \ref{eq:TwoExtremeHypRepeat}.  Then $\nu(A+B)=\nu(A+B+K_{1})=\nu(A+B+K_{2})$, since both $K_{i}$ have finite index.  Thus $A+B+K_{1}=A+B+K_{2}$.  Our assumption implies that $K_{1}$ is the stabilizer of $A+B+K_{1}$, and therefore $K_{1}$ is also the stabilizer of $A+B+K_{2}$, which is $K_{2}$.  Thus $K_{1}=K_{2}$.
\end{proof}

\section{Choquet decomposition}\label{sec:Choquet}

In order to derive Theorem \ref{th:MainOneMeanStrict} from Theorem \ref{th:OneExtreme}, we will write  a given invariant mean $m$ as a convex combination of extreme invariant means.  In the remainder of this article ``measure'' will refer only to countably additive measures.
\begin{definition}
Let $\Gamma$ be a discrete abelian group and $m\in \mathcal M(\Gamma)$.  We say that a probability measure $\sigma$ on $\mathcal M(\Gamma)$ \emph{represents $m$} if for each $f\in \ell^{\infty}(\Gamma)$, the map  $\Lambda_{f}:\mathcal M(\Gamma)\to \mathbb C$ given by $\Lambda_{f}(\nu)= \nu(f)$ is $\sigma$-measurable, and $m(f)=\int \nu(f)\, d\sigma(\nu)$.
\end{definition}
The following theorem is due to Bishop and de Leeuw, generalizing Choquet's theorem - see  \cite[pp.~22-23]{Phelps_LecturesOnChoquet2ndEd}.  We write $\operatorname{ex} X$ to denote the extreme points of the set $X$, and we write ``$\sigma$'' for the measure written as ``$\mu$'' in \cite{Phelps_LecturesOnChoquet2ndEd}.

\begin{theorem}\label{th:CBdL}  Suppose that $X$ is a compact convex subset of a locally convex space, and denote by $\mathcal S$ the $\sigma$-ring of subsets of $X$ which is generated by $\operatorname{ex}\,X $ and the Baire sets.  Then for each point $x_0$ in $X$ there exists a nonnegative measure $\sigma$ on $\mathcal S$ with $\sigma(X)=1$ such that $\sigma$ represents $x_0$ and $\sigma(\operatorname{ex}\, X)=1$.
\end{theorem}
Here ``$\sigma$ represents $x_0$'' means that for every continuous linear functional $\Lambda$ on $X$,
\begin{equation}\label{eq:Choquet}
\Lambda(x_0)=\int \Lambda(x)\, d\sigma(x).
\end{equation}
We call the measure $\sigma$ described in Theorem \ref{th:CBdL} a \emph{Choquet decomposition} of $x_{0}$.

A \emph{$\sigma$-ring} is a collection of sets closed under countable unions and set difference.

``Baire sets'' refers to the Baire subsets of $X$, these being elements of the smallest $\sigma$-ring containing the compact $G_{\delta}$ subsets of $X$.  Since $X$ itself is compact, this $\sigma$-ring is actually a $\sigma$-algebra.

The ``$\sigma$-ring of subsets of $X$ which is generated by $\operatorname{ex} X$ and the Baire sets'' is the smallest $\sigma$-ring containing $\operatorname{ex} X$ and the Baire subsets of $X$.

Corollary \ref{cor:ChoquetForMeans} below is the specialization of Theorem \ref{th:CBdL} to the case where $X=\mathcal M_{\tau}(\Gamma)$. In this setting $\operatorname{ex} X$ is the set $\mathcal M_{\tau}^{ext}(\Gamma)$ of extreme invariant means on $\ell^{\infty}(\Gamma)$.  The linear functionals we consider have the form $\Lambda_{f}$ described above, and in this setting (\ref{eq:Choquet}) becomes
\begin{equation}\label{eq:ChoquetMean}
	m(f) = \int \nu(f)\, d\sigma(\nu).
\end{equation} The $\sigma$-ring $\mathcal S$  matters only insofar as it contains $\mathcal M_{\tau}^{ext}(\Gamma)$ and ensures that for every $f\in \ell^{\infty}(\Gamma)$, the map $\Lambda_{f}$ is $\sigma$-measurable. For each $A$, $B\subseteq \Gamma$, this makes 
\[\{\nu\in \mathcal M_{\tau}^{ext}(\Gamma):\nu(A+B)<\nu(A)+\nu(B)\}\] $\sigma$-measurable, as the defining inequality is equivalent to $\nu(1_{A}+1_{B}-1_{A+B})>0$.

\begin{corollary}\label{cor:ChoquetForMeans}
	Let $\Gamma$ be a discrete abelian group and let $m\in \mathcal M_{\tau}(\Gamma)$.  Then there is a probability measure $\sigma$ on $\mathcal M_{\tau}(\Gamma)$ such that $\sigma(\mathcal M_{\tau}^{ext}(\Gamma))=1$, for all $f\in\ell^{\infty}(\Gamma)$ the map $\Lambda_{f}:\mathcal M_{\tau}(\Gamma)\to \mathbb C$ given by $\Lambda_{f}(\nu)= \nu(f)$ is $\sigma$-measurable, and  (\ref{eq:ChoquetMean}) holds.
\end{corollary}

\section{Invariant elements of \texorpdfstring{$L^{2}(m)$}{L2(m)}}\label{sec:Hilbert}

\subsection{Hilbert space associated to a mean} We summarize some facts from \S5.4 of \cite{Griesmer_DiscreteSumsets}.  

Let $\Gamma$ be a discrete abelian group.  Recall that $\mathcal M(\Gamma)$, $\ell^{\infty}(\Gamma)$, and $\mathcal M_{\tau}^{ext}(\Gamma)$ denote the set of means, invariant means, and extreme invariant means on  $\ell^{\infty}(\Gamma)$, respectively.

Let $m\in \mathcal M(\Gamma)$. For $f, g\in \ell^{\infty}(\Gamma)$, we write $f\sim_{m}g$ when $m(|f-g|^{2})=0$.  Note that $\sim_{m}$ is an equivalence relation; we write $[f]$ for the $\sim_{m}$-equivalence class of $f$ and we write $Z$ for the subspace $\{f\in \ell^{\infty}(\Gamma):m(|f|^{2})=0\}$.  We let $L^{2}(m)_{0}$ denote the quotient $\ell^{\infty}(\Gamma)/Z$. Note that $L^{2}(m)_{0}$ is a metric space under the metric $d$ given by $d([f],[g]):=m(|f-g|^{2})^{1/2}$.    We let $L^{2}(m)$ denote the metric completion of $L^{2}(m)_{0}$. A sequence $(f_{n})_{n\in \mathbb N}$ of bounded functions $f_{n}:\Gamma\to \mathbb C$ is \emph{$L^{2}(m)$-Cauchy} if $\lim_{n\to\infty} \sup_{i,j\geq n} m(|f_{i}-f_{j}|^{2})^{1/2}=0$; we may simply write ``Cauchy sequence'' if the context is clear.  Every element of $L^{2}(m)$ is represented by such a Cauchy sequence.  Given two elements $v, w\in L^{2}(m)$ represented by Cauchy sequences $(f_{n})_{n\in \mathbb N}$, $(g_{n})_{n\in \mathbb N}$, let $\langle v,w\rangle_{L^{2}(m)}:=\lim_{n\to\infty} m(f_{n}\bar{g}_{n})$. Then $\langle \cdot, \cdot \rangle_{L^{2}(m)}$ is an inner product generating the metric on $L^{2}(m)$.

When $v$, $w\in L^{2}(m)$ are represented by uniformly bounded Cauchy sequences $(f_{n})_{n\in \mathbb N}$ and $(g_{n})_{n\in \mathbb N}$, the sequence $(f_{n}g_{n})_{n\in \mathbb N}$ is a Cauchy sequence, as well; we write $vw$ for the element of $L^{2}(m)$ given by $(f_{n}g_{n})_{n\in \mathbb N}$. It is easy to verify that $vw$ does not depend on the choice of sequence representing $v$ or $w$.

For $f\in \ell^{\infty}(\Gamma)$ and $w\in L^{2}(m)$, we will write $w\sim_{m} f$ to mean that $w$ is represented by the constant Cauchy sequence $f_{n}=f$.

Note that $m$ extends to $L^{2}(m)$, via the definition $m(w):=\langle w,1_{\Gamma}\rangle_{L^{2}(m)}$, or simply $m(w):=\lim_{n\to\infty} m(f_{n})$, where $(f_{n})_{n\in \mathbb N}$ represents $w$.

\subsection{Translation action}
Recall the action $\tau$ on $\ell^{\infty}(\Gamma)$ from \S\ref{sec:Means}: $(\tau_{\gamma}f)(x):=f(x-\gamma)$.  Each $\tau_{\gamma}$ preserves the $L^{2}(m)$ distance: $\|\tau_{\gamma}f-\tau_{\gamma}g\|_{L^{2}(m)}=\|f-g\|_{L^{2}(m)}$.  Since $\ell^{\infty}(\Gamma)$ forms a dense subset of $L^{2}(m)$, we see that each $\tau_{\gamma}$ induces an isometry of $L^{2}(m)$: if $w\in L^{2}(m)$ is represented by an $L^{2}(m)$-Cauchy sequence $(f_{n})_{n\in \mathbb N}$, then $\tau_{\gamma}w$ is represented by $(\tau_{\gamma} f_{n})_{n\in \mathbb N}$.  We say $w\in L^{2}(m)$ is \emph{$\tau$-invariant} if $\tau_{\gamma}w=w$ for all $\gamma \in w$.

\subsection{Indicator elements}

\begin{definition}\label{def:Indicator}
	Let $m\in \mathcal M(\Gamma)$.  We say $w\in L^{2}(m)$ is an \emph{indicator element} of $L^{2}(m)$ if there is a sequence of sets $C_{n}\subseteq \Gamma$ such that $w=\lim_{n\to\infty} 1_{C_{n}}$ in $L^{2}(m)$.
\end{definition}
Definition \ref{def:StrongAbsoluteContinuity} can be abbreviated as follows.
\begin{definition}\label{def:Abbreviated}
	Let $m$, $m'\in \mathcal M(\Gamma)$ and $\beta\geq 0$.  We say that $m'$ is \emph{a $\beta$-restriction of $m$} if there is an indicator element $v\in L^{2}(m)$ with $m(v)\geq \beta$ such that $m'(f)=m(v)^{-1}m(fv)$ for all $f\in \ell^{\infty}(\Gamma)$.
\end{definition}

\subsection{Completeness}\label{sec:Completeness}

If it were case that for every $m\in \mathcal M_{\tau}(\Gamma)$, $L^{2}(m)_{0}$ is complete under the $L^{2}(m)$ metric, then Definition \ref{def:StrongAbsoluteContinuity} and the proofs of Theorems \ref{th:MainOneMeanStrict} and \ref{th:MainFolnerStrict} could be substantially simplified: the Cauchy sequence $(1_{C_{n}})_{n\in \mathbb N}$ in Definition \ref{def:Indicator} would converge to a $\sim_{m}$ equivalence class of an actual function $1_{C}\in \ell^{\infty}(\Gamma)$.  But D.~Fremlin, in Theorem 6 of \S4 of  \cite{BlassEtAl_NoteOnExtensions}, constructed invariant means $m$ on $\ell^{\infty}(\mathbb Z)$ such that $L^{2}(m)_{0}$ is not complete. Theorem 5.1 of \cite{Kunisada_DensityMeasures}  and related results therein  characterize certain constructions resulting in $L^{2}(m)_{0}$ being complete -- c.f.~\cite{Kunisada_AdditiveProperty}.

\subsection{Means from Choquet decomposition; outline of Proof of Theorem \ref{th:MainOneMeanStrict}}

Given an invariant mean $m$ on $\ell^{\infty}(\Gamma)$, a Choquet decomposition $\sigma$ of $m$, and a set $E\subseteq \mathcal M_{\tau}^{ext}(\Gamma)$ with $\sigma(E)>0$, we will define a new invariant mean $m'$ by
\begin{equation}\label{eq:NewMean}
	m'(f):=\frac{1}{\sigma(E)}\int \nu(f)1_{E}(\nu)\, d\sigma(\nu).
\end{equation}
Intuitively, there should be a $C\subseteq \Gamma$ such that $m'$ is also  given by $m'(f)=m(1_{C})^{-1}m(f1_{C})$, with $m(C)\geq \sigma(E)$ and $m(C\triangle (C+\gamma))=0$ for all $\gamma \in \Gamma$. 

While such a $C$ may not exist in light of \S\ref{sec:Completeness},   the following conjecture is plausible.

\begin{conjecture}\label{conj:Indicator}
	If $\Gamma$ is a discrete abelian group, $m\in \mathcal M_{\tau}(\Gamma)$, $\sigma$ is a Choquet decomposition of $m$, and $E\subseteq \mathcal M_{\tau}^{ext}(\Gamma)$ is $\sigma$-measurable, then there is a $\tau$-invariant indicator element $w\in L^{2}(m)$ such that $m(fw)=\int \nu(f)1_{E}(\nu)\, d\sigma(\nu)$ for all $f\in \ell^{\infty}(\Gamma)$.
\end{conjecture}
 As we do not have a proof of Conjecture \ref{conj:Indicator}, we are forced to deal with some additional technicalities outlined in Remark \ref{rem:Sad}. Conjecture \ref{conj:Indicator} would yield a short proof of Theorem \ref{th:MainOneMeanStrict}; we outline it here to motivate the subsequent technicalities.

\begin{proof}[Proof of Theorem \ref{th:MainOneMeanStrict}, assuming Conjecture \ref{conj:Indicator}] Let $A$, $B\subseteq \Gamma$ satisfy $\delta:=m(A+B)-m(A)-m(B)>0$, and let $\sigma$ be a Choquet decomposition of $m$.  Let $K\leq \Gamma$ be the finite index subgroup given by Theorem \ref{th:OneExtreme}, and write $k$ for its index. We aim to prove that there is an indicator element $w$ such that the mean given by $m'(f):=m(fw)$ satisfies
\begin{enumerate}
	\item\label{item:kdeltaStrong} $m'$ is a $k\delta$-restriction of $m$;
	\item\label{item:Relized*}  $m'(A+B)=m'(A+B+K)$;
	\item\label{item:SameGap} $m'(A)+m'(B)-m'(A+B)\geq \delta.$
\end{enumerate}

Let $E:=\{\nu \in \mathcal M_{\tau}^{ext}(\Gamma):\nu(A)+\nu(B)-\nu(A+B)>0\}$.  Note that
\begin{equation}\label{eq:int}
	\begin{split}	\int \bigl(\nu(A)+\nu(B)-\nu(A+B)\bigr)1_{E}(\nu)\, d\sigma(\nu)&\geq 	\int \nu(A)+\nu(B)-\nu(A+B)\, d\sigma(\nu)\\
		&=m(A)+m(B)-m(A+B).
	\end{split}
\end{equation}For each $\nu\in E$, we have $\nu(A+B)=\nu(A+B+K)=\dens^{*}(A+B+K)$, where $K$ is the subgroup given by Theorem \ref{th:OneExtreme}. Thus
\begin{equation}
	\int \nu(A+B)1_{E}(\nu)\, d\sigma(\nu)= \int \dens^{*}(A+B+K)1_{E}(\nu)\, d\sigma(\nu)=\dens^{*}(A+B+K)\sigma(E).
\end{equation}
Assuming Conjecture \ref{conj:Indicator}, we choose a $\tau$-invariant indicator element $w\in L^{2}(m)$ such that 
\begin{equation}\label{eq:GoodIndicator}
	m(fw)=\int \nu(f)1_{E}(\nu)\, d\sigma(\nu)
\end{equation} for all $f\in \ell^{\infty}(\Gamma)$.  We define $m'\in \mathcal M_{\tau}(\Gamma)$ by
\[
m'(f):=m(w)^{-1}m(fw).
\]
With $f=1_{\Gamma}$, $f=1_{A+B}$, and $f=1_{A}+1_{B}-1_{A+B}$, respectively, (\ref{eq:GoodIndicator}) becomes
\begin{align}
	\label{eq:mw}m(w)&=\sigma(E)\\
	\label{eq:m1ABw} m(1_{A+B}w)&= \dens^{*}(A+B+K)\sigma(E)\\
	\label{eq:m1A1B1AB} m((1_{A}+1_{B}-1_{A+B})w) &= \int \bigl(\nu(A)+\nu(B)-\nu(A+B)\bigr)1_{E}(\nu)\, d\sigma(\nu).
\end{align}
Combining (\ref{eq:mw}) and (\ref{eq:m1ABw}), we get
\begin{equation}\label{eq:mwd*}
	m'(A+B)=m(w)^{-1}m(1_{A+B}w)=\dens^{*}(A+B).
\end{equation}

To verify \ref{item:kdeltaStrong}, it suffices to prove $m(w)\geq k\delta$, where $k=[\Gamma:K]$. To do so, apply Corollary \ref{cor:Q} to get $k^{-1}\geq m'(1_{A}+1_{B}-1_{A+B})$. From (\ref{eq:int}) and (\ref{eq:GoodIndicator}) we get \begin{equation}\label{eq:m'Something} m'(1_{A}+1_{B}-1_{A+B})\geq m(w)^{-1}\delta.
\end{equation}   Combining these, we have $k^{-1}\geq m(w)^{-1}\delta$, so $m(w)\geq k\delta$.

Let $m'\in \mathcal M_{\tau}(\Gamma)$ be given by $m'(f)=m(w)^{-1}m(fw)$.  Then (\ref{eq:mwd*}) implies $m'(A+B)=\dens^{*}(A+B)=\dens^{*}(A+B+K)$.  Thus $m'(A+B)=m'(A+B+K)$, verifying \ref{item:Relized*}.  

Inequality \ref{item:SameGap} follows immediately from (\ref{eq:m'Something}), as $0<m(w)\leq 1$.
This completes the proof under the assumption that Conjecture \ref{conj:Indicator} holds. \end{proof}

\begin{remark}\label{rem:Sad} Unfortunately we cannot use Conjecture \ref{conj:Indicator}, as we have been unable to prove there is an \emph{indicator element} $w$ satisfying $m(fw)=\int \nu(f)1_{E}(\nu)\, d\sigma(\nu)$ for all $f\in \ell^{\infty}(\Gamma)$.  But it is easy to find functions $g_{n}:\Gamma\to [0,1]$ so that $w=\lim_{n\to\infty} g_{n}$ satisfies this equation.   With such $g_{n}$, we set $v_{n}=1_{\{g_{n}>0\}}$ and $v=\lim_{n\to\infty} v_{n}$. Then it is easy to prove conclusion \ref{item:kdeltaStrong} in Theorem \ref{th:MainOneMeanStrict}, but \ref{item:SameGap} may fail with $v_{n}$ so defined, so we replace $\{g_{n}>0\}$ with appropriately chosen subsets.  Lemma \ref{lem:Final} is the resulting modification of Conjecture \ref{conj:Indicator}.  Note that we are not assuming $\sigma$ is a Choquet decomposition in the hypothesis.
\end{remark}

\begin{lemma}\label{lem:Final}
	Let $\sigma$ be a probability measure on $\mathcal M_{\tau}(\Gamma)$ representing a mean $m$. Let $F:\mathcal M_{\tau}(\Gamma)\to [0,1]$ be $\sigma$-measurable and let $A$, $B\subseteq \Gamma$ satisfy
	\begin{equation}\label{eq:BigIntegral} \int (\nu(A)+\nu(B)-\nu(A+B))F(\nu)\, d\sigma:=\beta>0
	\end{equation}
	and
	\begin{equation}
		\int \nu(A+B)F(\nu)\, d\sigma(\nu) = \dens^{*}(A+B)\int F\, d\sigma.
	\end{equation}  Then there is a $\tau$-invariant indicator element $v\in L^{2}(m)$ such that 
	\begin{enumerate} 
		\item $m((1_{A}+1_{B}-1_{A+B})v)\geq \beta$;
		\item $m(v)^{-1}m(1_{A+B}v)=\dens^{*}(A+B)$. 
	\end{enumerate}
\end{lemma}

We prove Lemma \ref{lem:Final} at the end of \S\ref{sec:AbsCont}.

\subsection{Level sets} Due to the possible incompleteness mentioned in \S\ref{sec:Completeness}, a given element $w\in L^{2}(m)$ may not satisfy $w\sim_{m} f$ for any function $f:\Gamma\to \mathbb C$.  For such $w$, it is convenient to define associated elements of $L^{2}(m)$ which play the role of level sets.

Given $w\in L^{2}(m)$, a real-valued Cauchy sequence $(f_{n})_{n\in \mathbb N}$ representing $w$, and $E\subseteq \mathbb R$, consider
\begin{equation} C(f_{n},E):=\{\gamma \in \Gamma : f_{n}(\gamma) \in E\}.
\end{equation}  
It is natural to define the ``level set'' $\{a\leq w\leq b\}$ to be the limit of $g_{n}=1_{C(f_{n},[a,b])}$ in $L^{2}(m)$.  But $g_{n}$ so defined may not form an $L^{2}(m)$-Cauchy sequence, and even if it does, it is not clear whether the limit depends on the Cauchy sequence representing $w$.  The next lemma addresses these issues.

\begin{definition} We say that $w\in L^{2}(m)$ is \emph{real-valued} if there is an  $L^{2}(m)$-Cauchy sequence $f_{n}:\Gamma\to \mathbb R$ representing $w$.  We call such a sequence $(f_{n})_{n\in \mathbb N}$ real-valued, as well.  We write $L^{2}_{\mathbb R}(m)$ for the real-valued elements of $L^{2}(m)$.
\end{definition}
\begin{lemma}\label{lem:LevelSetInL2}
	Let $\Gamma$ be a discrete abelian group and $m\in \mathcal M(\Gamma)$.  For each  $w\in L_{\mathbb R}^{2}(m)$, there is a countable set $\Omega(w) \subseteq \mathbb R$ such that for all $a<b\in \mathbb R\setminus \Omega(w)$, every real-valued $L^{2}(m)$-Cauchy sequence $(f_{n})_{n\in \mathbb N}$ representing $w$, and each of the intervals $E=[a,b]$, $(a,b]$, $[a,b)$, $(a,b)$, the sequence $1_{C(f_{n},E)}$ is $L^{2}(m)$-Cauchy.  Furthermore, 
	\begin{enumerate}
	\item\label{item:depend}	 $h_{E}:=\lim_{n\to\infty} 1_{C(f_{n},E)}$ depends on $w$, and not on the sequence representing $w$;
	
	 \item\label{item:endpts} $h_{[a,b]}=h_{[a,b)}=h_{(a,b]}=h_{(a,b)}$ in $L^{2}(m)$;
	
\item\label{item:tauInv} if $w$ is $\tau$-invariant, then so is $h_{E}$.
\end{enumerate}
\end{lemma}
We prove Lemma \ref{lem:LevelSetInL2} in \S\ref{sec:LevelProof}. For $w\in L_{\mathbb R}^{2}(m)$ represented by a real-valued Cauchy sequence $(f_{n})_{n\in \mathbb N}$ and a given $a<b\in \mathbb R\setminus \Omega(w)$, we define
\begin{equation}
	1_{\{a\leq w\leq b\}}:=\lim_{n\to\infty} 1_{C(f_{n},[a,b])}.
\end{equation}
Note that $m(f_{n}1_{\{a\leq f_{n} \leq b\}})\leq bm(1_{\{a\leq f_{n} \leq b\}})$. From such observations and Lemma \ref{lem:LevelSetInL2}, we conclude
\begin{equation}\label{eq:TrivialBound}
	a m(1_{\{a\leq w\leq b\}})	\leq m(w1_{\{a\leq w\leq b\}})\leq bm(1_{\{a\leq w\leq b\}}) \qquad \text{for all } a<b\in \mathbb R\setminus \Omega(w).
\end{equation}
The next lemma packages this estimate for subsequent applications.
\begin{lemma}\label{lem:Riemann}
	Let $w\in L_{\mathbb R}^{2}(m)$, let $0\leq a<b\in \mathbb R\setminus \Omega(w)$, and let $v=1_{\{a\leq w\leq b\}}$.  Then
	\begin{equation}\label{eq:Riemann1}
		m(v)-\frac{1}{b}m(wv)\leq \Big(\frac{b-a}{b}\Big)m(v),
	\end{equation}
and if $g\in \ell^{\infty}(\Gamma)$, then
\begin{equation}\label{eq:Riemann2}
		|m(gv)-\frac{1}{b}m(gwv)|\leq \|g\|_{\infty}\Big(\frac{b-a}{b}\Big)m(v).
\end{equation}
\end{lemma}
\begin{proof}
	To prove (\ref{eq:Riemann1}), note that \[bm(v)-m(wv)=m((b-w)v)\leq m((b-a)v).\]  The last inequality follows from (\ref{eq:TrivialBound}).  Dividing by $b$ yields (\ref{eq:Riemann1}). Inequality (\ref{eq:Riemann2}) follows by linearity: 
	\begin{align*}m(gv)-\frac{1}{b}m(gwv)=m(g(v-b^{-1}wv))&\leq \|g\|_{\infty}m(|v-b^{-1}wv|)\\
		&=\|g\|_{\infty}m(v-b^{-1}wv)\\
		&=\|g\|_{\infty}(m(v)-b^{-1}m(wv))\\
		&\leq \|g\|_{\infty}  \Big(\frac{b-a}{b}\Big)m(v),
	\end{align*}
	by (\ref{eq:Riemann1}).
\end{proof}

\begin{lemma}\label{lem:LevelToIndicator}
	Let $\Gamma$ be a discrete abelian group, $m\in \mathcal M(\Gamma)$, $w\in L_{\mathbb R}^{2}(m)$, and assume that for all $\delta>0$ with $1-\delta$, $1+\delta \notin \Omega(w)$,  $w= w1_{\{1-\delta \leq w \leq 1+\delta\}}$.  Then $w$ is an indicator element.
\end{lemma}

\begin{proof}
		 Fix $0< \delta< \delta_{0} < 1/2$ so that $1-\delta, 1+\delta, 1-\delta_{0}$, $1+\delta_{0} \notin \Omega(w)$.  For $c\in \{\delta,\delta_{0}\}$, let $h_{c}=1_{\{1-c\leq w \leq 1+c\}}$. We will prove that $w=h_{\delta_{0}}$.	First we claim that 
	\begin{equation}\label{eq:h1/2hdelta}
		m(|h_{\delta_{0}}-h_{\delta}|)=0.
	\end{equation}
	To prove this, note that $m(|wh_{\delta_{0}}-wh_{\delta}|)\geq (\delta_{0}-\delta)m(|h_{\delta_{0}}-h_{\delta}|)$, and $wh_{\delta}=wh_{\delta_{0}}$ in $L^{2}(m)$.  Also
	\begin{equation}\label{eq:wMinusg}
		w-h_{\delta_{0}}= w-  wh_{\delta} +wh_{\delta} - h_{\delta}+h_{\delta}-h_{\delta_{0}}.
	\end{equation}
	We have $w=wh_{\delta}$ by assumption, and $m(|wh_{\delta} - h_{\delta}|)\leq m(|w-1|h_{\delta})\leq \delta$.  Combining these with (\ref{eq:wMinusg}), (\ref{eq:h1/2hdelta}), and the triangle inequality, we get $m(|w-h_{\delta_{0}}|)\leq \delta$.  Letting $\delta\to 0$, we get $m(|w-h_{\delta_{0}}|)=0$.  Thus $w=h_{\delta_{0}}$ in $L^{2}(m)$.
\end{proof}

\subsection{Nonnegative elements}

We say $w\in L^{2}(m)$ is \emph{nonnegative}, and write $w\geq 0$, if there is an $L^{2}(m)$-Cauchy sequence $(f_{n})_{n\in \mathbb N}$ representing $w$ such that $f_{n}(\gamma)\geq 0$ for all $\gamma \in \Gamma$ and all $n\in \mathbb N$.  We write $v\leq w$ if $w-v$ is nonnegative.  We write $0\leq w\leq 1_{\Gamma}$ if there is an $L^{2}(m)$-Cauchy sequence of functions $f_{n}:\Gamma \to [0,1]$ representing $w$.

\begin{lemma}\label{lem:to01}
	With $\Gamma$, $m$, and $w$ as in Lemma \ref{lem:Nonneg}, we have $0\leq w\leq 1_{\Gamma}$ if and only if $w$ and $1_{\Gamma}-w$ are both nonnegative. 
\end{lemma}

\begin{proof}
	The implication $0\leq w\leq 1_{\Gamma}$ $\implies$ $w$ and $1_{\Gamma}-w$ are both nonnegative is straightforward.  To prove the converse, assume $w$ and $1_{\Gamma}-w$ are both nonnegative.  Then there are $L^{2}(m)$-Cauchy sequences $(f_{n})_{n\in \mathbb N}$ and $(g_{n})_{n\in \mathbb N}$ representing $w$ such that $f_{n}(\gamma)\geq 0$ and $g_{n}(\gamma)\leq 1$ for all $n\in \mathbb N$, $\gamma\in \Gamma$.  Defining $h_{n}:\Gamma\to [0,1]$ by $h_{n}(\gamma)=0$ if $g_{n}(\gamma)<0$, $h_{n}(\gamma)=1$ if $f_{n}(\gamma)>1$, and $h_{n}(\gamma)=f_{n}(\gamma)$ otherwise, it is easy to verify that $\|f_{n}-h_{n}\|_{L^{2}(m)}\leq \|f_{n}-g_{n}\|_{L^{2}(m)}$. Thus $(h_{n})_{n\in \mathbb N}$ also represents $w$. \end{proof}

\begin{lemma}\label{lem:Nonneg}
Let $\Gamma$ be a discrete abelian group,  $m\in \mathcal M(\Gamma)$, and $w\in L^{2}(m)$.	The following are equivalent.
	\begin{enumerate}
		\item[(i)] $w$ is nonnegative.
		\item[(ii)] $\langle f,w\rangle_{L^{2}(m)}\geq 0$ for every $f:\Gamma \to [0,1]$.
		\item[(iii)] $\langle v,w\rangle_{L^{2}(m)}\geq 0$ for every nonnegative $v\in L^{2}(m)$.
	\end{enumerate}
\end{lemma}
\begin{proof}
	(i) $\implies$ (ii) and (ii)$\implies$(iii) are immediate from the definitions.  To prove (iii)$\implies$(i), suppose $w$ satisfies (iii).  Let $(f_{n})_{n\in \mathbb N}$ be an $L^{2}(m)$-Cauchy sequence of real-valued functions representing $w$.  Consider the truncations $g_{n}$ defined by $g_{n}(\gamma)=f_{n}(\gamma)$ if $f_{n}(\gamma)<0$, $g_{n}(\gamma)=0$ if $f_{n}(\gamma)>0$.  We will prove that \begin{equation}\label{eq:gnorm} \lim_{n\to\infty} \|g_{n}\|_{L^{2}(m)}=0,
	\end{equation} which implies $w=\lim_{n\to\infty} f_{n}-g_{n}$.  Since $f_{n}-g_{n}$ is nonnegative, this proves (i) holds.  
	
	Note that $g_{n}=\psi\circ f_{n}$, where $\psi(x)=x$ if $x\leq 0$, $\psi(x)=0$ if $x>0$.  Since $\psi$ is Lipschitz and $(f_{n})_{n\in \mathbb N}$ is $L^{2}(m)$-Cauchy, we get that $g_{n}$ is $L^{2}(m)$-Cauchy, as well.  Let $v=\lim_{n\to\infty} -g_{n}$, so that $v$ is nonnegative. Condition (iii) then implies $\langle v,w\rangle_{L^{2}(m)} \geq 0$. We also have $-g_{n}(\gamma)f_{n}(\gamma)\leq 0$ for all $\gamma\in \Gamma$, so $m(-g_{n}f_{n})\leq 0$ for all $n$.  Thus $\langle v,w\rangle_{L^{2}(m)}\leq 0$. Now $-g_{n}f_{n}=|g_{n}|^{2}$, so we get $\lim_{n\to\infty} m(|g_{n}|^{2})=0$.  This implies (\ref{eq:gnorm}) and completes the proof.
\end{proof}

\subsection{Conservation of upper Banach density}

\begin{lemma}\label{lem:dStarConvex}  Let $\Gamma$ be a discrete abelian group, and 
 $m$, $m_{j}\in \mathcal M_{\tau}(\Gamma)$,  $c_{j}>0$,  satisfy $\sum_{j} c_{j}=1$ and $m = \sum_{j} c_{j}m_{j} $. If $D\subseteq \Gamma$  satisfies $m(D)=\dens^{*}(D)$ then $m_{j}(D)=\dens^{*}(D)$ for each $j$.
\end{lemma}

\begin{proof}
	Let $D$, $m$, and $m_{j}$ be as in the hypothesis.  By symmetry it suffices to prove that $m_{1}(D)=\dens^{*}(D)$. Since $\dens^{*}(D)=\sup \{\eta(D): \eta\in  \mathcal M_{\tau}(\Gamma)\}$, the hypothesis implies $m_{j}(D)\leq  \dens^{*}(D)$ for each $j$.  Thus $\dens^{*}(D)=\sum_{j} c_{j}m_{j}(D)\leq c_{1}m_{1}(D)+\sum_{j\geq 2} c_{j}\dens^{*}(D)$, which implies $\dens^{*}(D)-\sum_{j\geq 2} c_{j}\dens^{*}(D) \leq c_{1}m_{1}(D)$.  Thus $c_{1}\dens^{*}(D)\leq c_{1}m_{1}(D)$, which implies $\dens^{*}(D)\leq m_{1}(D)$.  By definition of $\dens^{*}$, this implies $m_{1}(D)=\dens^{*}(D)$.
\end{proof}

\begin{lemma}\label{lem:ExtremityOfdStar}
	Let $\Gamma$ be a discrete abelian group, $m\in \mathcal M_{\tau}(\Gamma)$, $D\subseteq \Gamma$.   If $w\in L^{2}(m)$ is $\tau$-invariant and nonnegative, and
	\begin{equation}\label{eq:wRealizesd*E}
		\langle 1_{D},w\rangle _{L^{2}(m)}=m(w)\dens^{*}(D),
	\end{equation} then for every $\beta > \alpha >0$, $\alpha, \beta \notin \Omega(w)$, we have $\langle 1_{D}, 1_{\{\alpha \leq w\leq \beta\}}\rangle_{L^{2}(m)} = m(1_{\{\alpha \leq w\leq \beta\}})\dens^{*}(D)$.
\end{lemma}

\begin{proof}
	When $m(w)=0$ the conclusion follows trivially, so assume $m(w)>0$. 	Let $m'$ be the invariant mean given by $m'(f):=m(w)^{-1}m(fw)$, so that
	\begin{equation}\label{eq:mPrimeExtreme}
		m'(1_{D})=\dens^{*}(D).
	\end{equation} 	Let $\beta > \alpha >0$, $\alpha, \beta \notin \Omega(w)$. Fix $a_{i}$ with $\alpha=a_{0}<a_{1}<\dots<a_{k}=\beta$ with $a_{j}\notin \Omega(w)$.  Let $I_{j}:=[a_{j-1},a_{j})$, 
	$I_{k}:=[a_{k-1},\beta]$, and $I_{k+1}=(\beta,\infty)$. Let 
	\[\mathcal P = \{I_{j}:j\leq k+1 \text{ and } m(1_{\{w\in I_{j}\}})>0\}\]
	
	Let $v_{I}= 1_{\{w\in I\}}$ and $w_{I}=wv_{I}$.  Note that 
	\begin{equation}\label{eq:sumIw}
		w=\sum_{I\in \mathcal P} w_{I}
	\end{equation} and $\sum_{I\in \mathcal P} v_{I}\leq 1_{\Gamma}$.
	
	For each $I\in \mathcal P$, let $m_{I}'$ be the invariant mean given by $m_{I}'(f)=m(w_{I})^{-1} \langle f, w_{I}\rangle$. Let $c_{I}:=m(w_{I})/m(w)$, so that $\sum_{I\in \mathcal P} c_{I}=1$.  Then (\ref{eq:sumIw}) implies  $m'(f) = \sum_{I\in \mathcal P} c_{I}m_{I}'$.  From (\ref{eq:mPrimeExtreme}) and Lemma \ref{lem:dStarConvex} we get $m_{I}(1_{D})=\dens^{*}(D)$ for each $I\in \mathcal P$.  In other words,
	\begin{equation}\label{eq:m1Ew}
		m(1_{D}w_{I})=m(w_{I})\dens^{*}(D).
	\end{equation}
	Let $a_{I}=\min I$, $b_{I}=\max I$, and let $v_{I}=1_{\{a_{I}\leq w\leq b_{I}\}}$, so that $\sum_{I\in \mathcal P} \leq m(1_{\alpha \leq w\leq \beta})\leq 1$.  Then (\ref{eq:Riemann1}) in Lemma \ref{lem:Riemann} implies
	\begin{equation}\label{eq:mvI}
		0\leq m(v_{I})-\frac{1}{b_{I}}m(w_{I})\leq \Big(\frac{b_{I}-a_{I}}{b_{I}}\Big)m(v_{I}).
	\end{equation}
	Let $\delta_{\mathcal P}=\max_{I\in \mathcal P} b_{I}-a_{I}$. We have 
	\[\sum_{I\in \mathcal P} \Big(\frac{b_{I}-a_{I}}{b_{I}}\Big)m(v_{I})\leq \max_{I\in \mathcal P} \frac{b_{I}-a_{I}}{b_{I}}\sum_{I\in \mathcal P} m(v_{I})\leq \frac{1}{\alpha}\delta_{\mathcal P}\cdot m(1_{\Gamma})= \frac{1}{\alpha}\delta_{\mathcal P}.\] 
	Multiplying inequality (\ref{eq:mvI}) by $\dens^{*}(D)$ and summing, we therefore have 
	\begin{equation}\label{eq:Approx1}
		0\leq \sum_{I \in \mathcal P} m(v_{I})\dens^{*}(D)-\frac{1}{b_{I}}m(1_{w_{I}})\dens^{*}(D) \leq  \frac{1}{\alpha}\delta_{\mathcal P}
	\end{equation}
	Likewise (replacing $w$ with $1_{D}w$), we have
	\begin{equation}\label{eq:Approx2}
		0\leq \sum_{I\in \mathcal P} m(1_{D}v_{I})-\frac{1}{b_{I}}m(1_{D}w_{I}) \leq \frac{\delta_{\mathcal P}}{\alpha}.
	\end{equation}
  Combining (\ref{eq:m1Ew}), (\ref{eq:Approx1}), and (\ref{eq:Approx2}), we get
	\begin{equation}
		\big|\sum_{I\in \mathcal P} m(v_{I})\dens^{*}(D)-m(1_{D}v_{I})|\leq (2/\alpha)\delta_{\mathcal P}.
	\end{equation}
This simplifies to $|m(1_{\{\alpha \leq w\leq \beta\}})\dens^{*}(D)-m(1_{D}1_{\{\alpha \leq w\leq \beta\}})|\leq (2/\alpha)\delta_{\mathcal P}$.  Letting $\delta_{\mathcal P}\to 0$, we get the desired conclusion. \end{proof}

\subsection{Restrictions}\label{sec:AbsCont}

\begin{lemma}\label{lem:AC}
	Let $\Gamma$ be a discrete abelian group, $m$, $m'\in \mathcal M(\Gamma)$ and $C>0$.  If $|m'(f)|\leq Cm(|f|)$ for all $f\in \ell^{\infty}(\Gamma)$, then there is a nonnegative element $w\in L^{2}(m)$ such that 
	\begin{equation}\label{eq:RadonNikodymW}
		m'(f)=\langle f,w\rangle_{L^{2}(m)} \quad \text{for all } f\in \ell^{\infty}(\Gamma).
	\end{equation}
	If we assume $m$, $m'\in \mathcal M_{\tau}(\Gamma)$, then $w$ above will be $\tau$-invariant.
\end{lemma}

\begin{proof}
Let $m$ and $m'$ be as in the hypothesis.	Note that $f\mapsto m'(f)$ is continuous in the $L^{2}(m)$-norm: if $m(|f|^{2})\leq 1$, then $|m'(f)|\leq m'(|f|)\leq m'(|f|^{2})^{1/2}m'(|1_{\Gamma}|^{2})^{1/2}\leq C^{1/2}m(|f|^{2})^{1/2}\leq C^{1/2}$.  Thus $f\mapsto m'(f)$ extends to a continuous linear functional $\Lambda$ on $L^{2}(m)$.  By the Reisz representation theorem for Hilbert spaces, there is a $w\in L^{2}(m)$ such that $m'(f)=\langle f,w\rangle_{L^{2}(m)}$ for all $f\in \ell^{\infty}(\Gamma)$.  Since $\langle f ,w\rangle_{L^{2}(m)} = m'(f)\geq 0$ for all $f:\Gamma \to [0,1]$, Lemma \ref{lem:Nonneg} implies $w$ is nonnegative.
	
If $m$, $m'\in \mathcal M_{\tau}(\Gamma)$ and $w\in L^{2}(m)$ satisfies (\ref{eq:RadonNikodymW}), we show that $w$ is $\tau$-invariant.  Under these assumptions, we have $\langle f,\tau_{\gamma}w\rangle_{L^{2}(m)}=\langle \tau_{\gamma}^{-1}f,w\rangle_{L^{2}(m)}=m'(\tau_{\gamma}^{-1}f)=m'(f)=\langle f, w\rangle_{L^{2}(m)}$ for all $f\in \ell^{\infty}(\Gamma)$ and all $\gamma\in \Gamma$.  Since $\ell^{\infty}(\Gamma)$ determines a dense subspace of $L^{2}(m)$, it follows that $\tau_{\gamma}w=w$ for all $\gamma\in \Gamma$.
\end{proof}

\begin{lemma}\label{lem:BoundRepresentative}
Let $\Gamma$ be a discrete abelian group, $m\in \mathcal M(\Gamma)$, and $\sigma$ a probability measure on $\mathcal M(\Gamma)$ representing $m$.  Let  $F:\mathcal M(\Gamma)\to [0,1]$ be $\sigma$-measurable. Then there is a $w\in L^{2}(m)$ with $0\leq w\leq 1_{\Gamma}$, such that
\begin{equation}\label{eq:fnRepresent}
\langle f,w\rangle_{L^{2}(m)}=\int \nu(f) F(\nu)\, d\sigma(\nu) \qquad \text{for all } f\in \ell^{\infty}(\Gamma).
\end{equation}
Furthermore, if $\sigma(\mathcal M_{\tau}(\Gamma))=1$, then $w$ is $\tau$-invariant.	
\end{lemma}
\begin{proof}
Let $c:=\int F(\nu)\, d\sigma(\nu)$, and define $m'\in \mathcal M(\Gamma)$ by $m'(f):=c^{-1}\int \nu(f)F(\nu) d\sigma(\nu)$.   Then $m'(|f|)\leq c^{-1}m(|f|)$ for all $f\in \ell^{\infty}(\Gamma)$.  By Lemma \ref{lem:AC}, there is a nonnegative $v\in L^{2}(m)$ such that $m'(f)=\langle f,v\rangle_{L^{2}(m)}$ for all $f\in \ell^{\infty}(\Gamma)$.  This means that $w:=cv$ satisfies (\ref{eq:fnRepresent}). We claim that $1_{\Gamma}-w$ is nonnegative as well.  To see this, note that when $f:\Gamma \to [0,1]$, we have \begin{align*}
	\langle f,1_{\Gamma}-w\rangle_{L^{2}(m)}	&= \langle f,1_{\Gamma} \rangle_{L^{2}(m)}- \langle f,w\rangle_{L^{2}(m)}\\
	&= \int \nu(f)\, d\sigma(\nu)-\int \nu(f)F(\nu)\, d\sigma(\nu)\\
	&=\int \nu(f)(1-F(\nu))\, d\sigma(\nu)\geq 0.
\end{align*} 
Thus $w$ and $1_{\Gamma}-w$ are both nonnegative elements of $L^{2}(m)$, so Lemma \ref{lem:to01} implies $0\leq w\leq 1_{\Gamma}$.

	To prove the last assertion, note that if $\sigma(\mathcal M_{\tau}(\Gamma))=1$, then $m$ is invariant, and the $m'$ defined above is invariant, as well.  Lemma \ref{lem:AC} then implies $w$ is $\tau$-invariant.
\end{proof}

\begin{lemma}\label{lem:StepApproximation}
Let $\Gamma$ be a discrete abelian group, $m\in \mathcal M_{\tau}(\Gamma)$, $D\subseteq \Gamma$, and $g\in \ell^{\infty}(\Gamma)$. Let $w\in L^{2}(m)$ be $\tau$-invariant and satisfy 
\begin{align}
	0&\leq w\leq 1_{\Gamma},\\
\beta&:= \langle g, w\rangle_{L^{2}(m)}>0,\\
	\label{eq:dStarRealized}	\dens^{*}(D)&=m(w)^{-1 } \langle 1_{D},w\rangle_{L^{2}(m)}.
\end{align}
   	Then there is an $\alpha\in (0,1)\setminus \Omega(w)$ and a $\tau$-invariant indicator element $v\in L^{2}(m)$ such that
	\begin{enumerate}
		\item\label{item:vSupport} $v\leq 1_{\{w\geq \alpha\}}$;
	\item\label{item:gvbeta}  $\langle g, v \rangle_{L^{2}(m)}\geq \beta$;
	\item\label{item:dStarE} $\dens^{*}(D)=m(v)^{-1}\langle 1_{D}, v\rangle_{L^{2}(m)}$.
\end{enumerate}
\end{lemma}
\begin{proof}	
	Let $\Gamma$, $m$, $w$, and $g$ be as in the hypothesis.  Let $\Omega=\Omega(w)$ be as in Lemma \ref{lem:LevelSetInL2}. 	
	
	Case 1: $w$ is an indicator element.  In this case we take $v=w$ and are done.
	
	Case 2: $w$ is not an indicator element. Then by Lemma \ref{lem:LevelToIndicator} and the assumption $w\leq 1_{\Gamma}$, there is a $c \in (0,1)\setminus \Omega$ such that $m(w1_{\{w\leq c\}})>0$.  Fix such a $c$, and let $w'=w1_{\{w\leq c\}}$, $w''=w1_{\{c<w\}}$.  Then $w'+w''=w$.
	
	We will prove that for some $\delta>0$, there are collections $\mathcal Q'$, $\mathcal Q''$ of mutually disjoint intervals bounded away from $0$, with endpoints not in $\Omega$, such that $I\subseteq (0,c)$ for $I\in \mathcal Q'$, $I\subseteq [c,1]$ for $I\in \mathcal Q''$, and
	\begin{align}\label{eq:FirstDelta} \sum_{I\in \mathcal Q'} m(g 1_{\{w\in I\}})&> m(gw')+\delta\\
		 \label{eq:SecondDelta}\sum_{I\in \mathcal Q''} m(g 1_{\{w\in I\}})&>m(gw'')-\delta.
	\end{align}
	Having found such collections, we will let $\mathcal Q=\mathcal Q'\cup \mathcal Q''$, and let $v=\sum_{I\in \mathcal Q} 1_{\{w\in I\}}$.  Adding (\ref{eq:FirstDelta}) to (\ref{eq:SecondDelta}), we get $m(gv)> m(v(w'+w''))=m(gw)$, so conclusion \ref{item:gvbeta} is satisfied.    For an interval $I\subseteq [0,1]$, let $v_{I}:= 1_{\{w\in I\}}$ and let $w_{I}:=w 1_{\{w\in I\}}$.

	To prove Conclusion \ref{item:dStarE}, we apply Lemma \ref{lem:ExtremityOfdStar} to get $m(1_{D}v_{I})=m(v_{I})\dens^{*}(D)$ for each $I\in \mathcal Q$.  Summing, we get $m(v1_{D})=m(v)\dens^{*}(D)$, so \ref{item:dStarE} holds.  The $\tau$-invariance of $v$ follows from Lemma \ref{lem:LevelSetInL2}. 
	
	We now construct the collections $\mathcal Q'$ and $\mathcal Q''$.	Let $\varepsilon>0$, and choose $\alpha>0$ such that $m(|g|w 1_{\{w\leq \alpha\}})<\varepsilon m(|g|w')$.
	
		Let $\alpha=a_{1}<\cdots<a_{k}=c<a_{k+1}<\dots<a_{r}=1$, $a_{j}\in [0,1]\setminus \Omega$, with $\max_{j} a_{j}-a_{j-1}<\alpha \varepsilon $.  Set $I_{j}=[a_{j-1},a_{j})$, $j=1,\dots,r-1$, $I_{r}=[a_{r-1},a_{r}]$. Then $\mathcal P = \{I_{j}:j=1,\dots,r\}$ is a partition of $[0,1]$.

Let $a_{I}=\inf I$, $b_{I}=\sup I$.  Inequality (\ref{eq:Riemann2}) in Lemma \ref{lem:Riemann} implies
   \begin{equation}\label{eq:mgvi}
   	|m(gv_{I})-\frac{1}{b_{I}}m(g w_{I})| \leq (\varepsilon/b_{I}) m(v_{I})\leq \varepsilon m(v_{I}). 
   \end{equation}
   Let $\mathcal Q' = \{I_{j}: j\leq k, m(gw_{I})>0\}$ and $\mathcal Q''=\{I_{j}: k<j\leq r, m(gw_{I})>0\}$.  Then 
   \[\sum_{I\in \mathcal Q'} m(gw_{I})\geq m(gw')-m(gw'1_{w\leq \alpha})\geq m(gw')- \varepsilon m(|g|w'),\]  so $\sum_{I\in \mathcal Q} \frac{1}{b_{I}}m(gw_{I})\geq \frac{1-\varepsilon}{c}m(gw')$. Inequality  (\ref{eq:mgvi}) now implies
  \[
  \sum_{I\in \mathcal Q'} m(gv_{I}) \geq \frac{1-\varepsilon}{c}m(gw_{I}) - \varepsilon\sum_{I\in \mathcal Q'}m(v_{I}) \geq \frac{1-\varepsilon}{c}m(gw') -\varepsilon.
  \]
Since $c<1$ is fixed, we can make $\varepsilon$ sufficiently small that $\frac{1-\varepsilon}{c}>1$, and also small enough that $\frac{1-\varepsilon}{c}m(g w')-\varepsilon> m(g w')$.  This proves (\ref{eq:FirstDelta}). Inequality (\ref{eq:SecondDelta}) is proved similarly, with $1$ in place of $c$. 
\end{proof}
\begin{proof}[Proof of Lemma \ref{lem:Final}]
	With $A$, $B$, $m$, and $F$ as in the hypothesis, let $g=1_{A}+1_{B}-1_{A+B}$, so that the integral in (\ref{eq:BigIntegral}) is $ \int \nu(g)F(\nu)\, d\sigma$.  Lemma \ref{lem:BoundRepresentative} provides a $w\in L^{2}(m)$ with $0\leq w\leq 1_{\Gamma}$ satisfying $m(gw)=\beta$ and $m(1_{A+B}w)=\dens^{*}(A+B)m(w)$. 
	Applying Lemma \ref{lem:StepApproximation} with $A+B$ in place of $E$ yields the desired $v$.	  \end{proof}

\subsection{Extracting a F{\o}lner net from an invariant vector.}

The next lemma is needed to deduce Proposition \ref{prop:FullFolnerStrict} from Theorem \ref{th:MainOneMeanStrict}.  Recall that a F{\o}lner net $\mb \Phi$ is (by definition) full if $m(f):=\lim_{i\in I}\frac{1}{|\Phi_{i}|}\sum_{\gamma\in \Phi_{i}} f(\gamma)$ is well-defined for all $f\in \ell^{\infty}(\Gamma)$, and therefore defines an invariant mean.

\begin{lemma}\label{lem:CauchyAndInvariantSetsToFolner}
	Let $\Gamma$ be a discrete abelian group and let $m\in \mathcal M_{\tau}(\Gamma)$ be given by a full F{\o}lner net $(\Phi_{i})_{i\in I}$.   Let $v\in L^{2}(m)$ be a $\tau$-invariant indicator element with $c:=m(v)>0$.  Then there are subsets $\Psi_{i}\subseteq \Phi_{i}$ such that $(\Psi_{i})_{i\in I}$ is a full F{\o}lner net and
	\begin{equation}\label{eq:PsiRepresentsV}
		c^{-1}m(fv)=\lim_{i\in I} \frac{1}{|\Psi_{i}|}\sum_{\gamma\in \Psi_{i}} f(\gamma) \quad \text{for all } f\in \ell^{\infty}(\Gamma).
	\end{equation}
	Consequently $\lim_{i\in I} |\Psi_{i}|/|\Phi_{i}|=m(v)$.
\end{lemma}

\begin{proof} 	Write $v$ as $\lim_{n\to\infty} 1_{C_{n}}$ in $L^{2}(m)$. 	For each $i\in I$, we will choose some $n_{i}\in \mathbb N$ and set $\Psi_{i}=\Phi_{i}\cap C_{n_{i}}$.  
	
 	We assume $\Gamma$ has infinite cardinality, as otherwise the lemma is trivial. Let $\Gamma_{1}\subseteq \Gamma_{2}\subseteq \dots $ be a strictly increasing sequence of finite subsets of $\Gamma$.  We are \emph{not} assuming $\Gamma = \bigcup_{k\in \mathbb N} \Gamma_{k}$.
	
	Let $(\varepsilon_{j})_{j\in \mathbb N}$ be a sequence of positive numbers converging to $0$.

		Choose $n_{1}$ so that $\big|c-\lim_{i\in I} |C_{n}\cap \Phi_{i}|/|\Phi_{i}|\big|<\varepsilon_{1}$ for all $n\geq n_{1}$.
	
	Choose $i_{1}\in I$ such that 	 $\big |c-|C_{n_{1}}\cap \Phi_{i}|/|\Phi_{i}|\big |<\varepsilon_{1}$ for all $i\succeq i_{1}$.

	Suppose $k\in \mathbb N$, and $n_{1}<n_{2}<\dots< n_{k-1}$, $i_{1}\preceq i_{2}<\preceq \dots \preceq i_{k-1}$ are defined.  Choose $n_{k}>n_{k-1}$ so that $\big |c-\lim_{i\in I} |C_{n}\cap \Phi_{i}|/|\Phi_{i}|\big|<\varepsilon_{k}$ for all $n\geq n_{k}$.
	
	Choose $i_{k}\in I$ so that $i_{k}\succeq i_{k-1}$ and each of the following hold:
	\begin{equation}\label{eq:ikBound}
		\big|m(C_{n})-\frac{1}{|\Phi_{i}|}|C_{n}\cap \Phi_{i}|\big| <\varepsilon_{k} \quad \text{for all } n\leq n_{k},\ i\succeq i_{k},
	\end{equation}
	\begin{equation}\label{eq:CnCnprime}
		|C_{n}\cap C_{n'} \cap \Phi_{i}|/|\Phi_{i}|\geq m(C_{n}\cap C_{n'})- \varepsilon_{k} \text{ for all } n, n'\leq n_{k}, i\succeq i_{k},
	\end{equation}
	\begin{equation}\label{eq:ExploitFolner}
	|\Phi_{i}\cap (\Phi_{i}+\gamma)|/|\Phi_{i}|>(1-\varepsilon_{k}) \qquad \text{for all } i\succeq i_{k}, \gamma \in \Gamma_{k}.
	\end{equation}
	
	For each $i\in I$, let $\Psi_{i}=C_{n_{k}}\cap \Phi_{i}$ if $i\succeq i_{k}$ and $i\nsucceq i_{k+1}$.  Let $\Psi_{i} = \Phi_{i}$ if $i\nsucceq i_{1}$.

\begin{observation}\label{obs:Cofinal} Condition (\ref{eq:ExploitFolner}) implies that no $i\in I$ satisfies $i\succeq i_{k}$ for all $k\in \mathbb N$: otherwise finiteness of $\Phi_{i}$ and (\ref{eq:ExploitFolner}) imply $\Phi_{i}=\Phi_{i}+\gamma$ for all $\gamma \in \bigcup_{k} \Gamma_{k}$, which is impossible when $\Phi_{i}$ is a finite set. Thus we will have the following dichotomy: 
	\begin{equation}\label{eq:Cofinal?}
		\text{If $i\in I$, then either  there is a $k\in \mathbb N$ such that $i_{k}\preceq i$ and $i_{k+1}\npreceq i$, or $i_{1}\npreceq i$.}
	\end{equation}
Thus, to prove that some net $(y_{i})_{i\in I}$ indexed by $I$ satisfies $\lim_{i\in I} y_{i}=y$, it suffices to prove that $\lim_{k\to\infty} \sup_{i_{k}\preceq i, i_{k+1}\npreceq i} |y-y_{i}|=0$. Note that it does not suffice to prove that $\lim_{k\to\infty} |y-y_{i_{k}}|=0$, as (\ref{eq:Cofinal?}) does not imply $\{i_{k}:k\in \mathbb N\}$ is cofinal in $(I,\preceq)$.
\end{observation}

	We claim that 
	\begin{equation}\label{eq:PsiBound}
		\lim_{i\in I} |\Psi_{i}|/|\Phi_{i}|=c.
	\end{equation}
	To prove this, fix $\varepsilon>0$ and choose $N$ large enough that $c-\varepsilon<m(C_{n})<c+\varepsilon$ for all $n\geq N$. For sufficiently large $k$, and $i\in I$ such that $i\succeq i_{k}$,  $i\nsucceq i_{k+1}$, (\ref{eq:ikBound}) implies
	\[c-(\varepsilon +\varepsilon_{k})<|\Psi_{i}|/|\Phi_{i}|=|C_{n_{k}}\cap \Phi_{i}|/|\Phi_{i}|<c+\varepsilon+\varepsilon_{k}.\] 
	 To prove (\ref{eq:PsiRepresentsV}) let $\varepsilon>0$ and $f:\Gamma \to [0,1]$. It suffices to prove that for all sufficiently large $n$,  
	 \begin{equation}\limsup_{i\in I} \Big|\frac{1}{|\Psi_{i}|}\sum_{{\gamma\in \Psi_{i}}} f(\gamma)-\frac{1}{c|\Phi_{i}|}\sum_{i\in \Phi_{i}}  1_{C_{n}}(\gamma)f(\gamma)\Big|<\varepsilon.
\end{equation}
The expression in absolute values on the left simplifies as
\begin{equation}
	\frac{1}{|\Phi_{i}|}\sum_{\gamma\in \Phi_{i}} \Big(\frac{|\Phi_{i}|}{|\Psi_{i}|}1_{\Psi_{i}}(\gamma)-c^{-1}1_{C_{n}}(\gamma)\Big)f(\gamma).
\end{equation}
In light of (\ref{eq:PsiBound}), it suffices to prove that when $n$ is sufficiently large, we have
\begin{equation}\label{eq:PsiCbound}
	c-\varepsilon\leq \liminf_{i\in I} |\Psi_{i}\cap C_{n}|/|\Phi_{i}|\leq  \limsup_{i\in I} |\Psi_{i}\cap C_{n}|/|\Phi_{i}|\leq c+\varepsilon.
\end{equation}	
	The rightmost inequality follows from (\ref{eq:PsiBound}) and the middle inequality is trivial. To prove the leftmost inequality, choose $N_{\varepsilon}$ so that $\inf_{n, n'\geq N_{\varepsilon}}m(C_{n}\cap C_{n'})> c-\varepsilon$. Fix $k\in \mathbb N$ so that $\varepsilon_{k}<\varepsilon$ and $n_{k}\geq N_{\varepsilon}.$  We will prove that if  $N_{\varepsilon} \leq n \leq n_{k}$ and $i\succeq i_{k}$, then
	\begin{equation}\label{eq:PsiiCn}
		|\Psi_{i}\cap C_{n}|/|\Phi_{i}| \geq m(C_{n})-\varepsilon_{k}.
	\end{equation}
	So fix such an $n$. 	Let $j\in \mathbb N$ with $j\geq k$.  If $i\succeq i_{j}$, and $i\nsucceq i_{j+1}$ we have $\Psi_{i}=C_{n_{j}}\cap \Phi_{i}$. Then (\ref{eq:CnCnprime}) implies
	\[|\Psi_{i}\cap C_{n}|/|\Phi_{i}|=|C_{n_{j}}\cap C_{n}\cap \Phi_{i}|/|\Phi_{i}|\geq m(C_{n_{j}}\cap C_{n})-\varepsilon_{j}.\]  	
  Our choice of $N_{\varepsilon}$ implies the right-hand side above is at least $c-\varepsilon$. This completes the proof of (\ref{eq:PsiCbound}).

	We now show that  $\mb \Psi$ is a F{\o}lner net.  Fix $\gamma\in \Gamma$, let $\varepsilon>0$, and choose $N$  large enough that 
	\begin{equation}\label{eq:CnDelta}
		d_{\mb \Phi}(C_{n}\triangle (C_{n}+\gamma))<\varepsilon 
	\end{equation}
	\begin{equation} c-\varepsilon< d_{\mb\Phi}(C_{n}) <c+\varepsilon,
		\end{equation}
		 \begin{equation}
		 	\lim_{i\in I} |\Psi_{i}\cap C_{n}|/|\Psi_{i}| > 1-\varepsilon
		 \end{equation} each hold for all $n\geq N$.	Since $\Psi_{i}\subseteq \Phi_{i}$ for each $i$, these and (\ref{eq:PsiBound}) imply
		 \begin{equation}
		\lim_{i\in I}\frac{1}{|\Phi_{i}|}|(\Psi_{i}\triangle C_{n})\cap \Phi_{i}|<2\varepsilon \quad \text{for all } n\geq N.
	\end{equation}	 
	Together with (\ref{eq:CnDelta}), this implies  $\lim_{i\in I}\frac{1}{|\Phi_{i}|}|\Psi_{i}\triangle (\Psi_{i}+\gamma)|<4\varepsilon$.	Thus \[\lim_{i\in I} \frac{1}{|\Psi_{i}|}|\Psi_{i}\triangle (\Psi_{i}+\gamma)|=\lim_{i\in I}\frac{|\Phi_{i}|}{|\Psi_{i}|}\frac{1}{|\Phi_{i}|}|\Psi_{i}\triangle (\Psi_{i}+\gamma)|<4\varepsilon/c. \]
Since $\varepsilon>0$ and $\gamma \in \Gamma$ were arbitrary, this shows that $\mb \Psi$ is a F{\o}lner net.
\end{proof}

\section{Proof of Theorem \ref{th:MainOneMeanStrict} and Proposition \ref{prop:FullFolnerStrict}}\label{sec:MainProofs}

\subsection{Proof of Theorem \ref{th:MainOneMeanStrict}}

Recall the statement:

Let $\Gamma$ be a discrete abelian group and $A$, $B\subseteq \Gamma$ nonempty, and let $K\leq \Gamma$ be the KJ-stabilizer of $A+B$ given by Theorem \ref{th:OneExtreme}, so that $k:=[\Gamma:K]$ is finite and $K$ is the stabilizer of $A+B+K$.

If $m\in \mathcal M_{\tau}(\Gamma)$ satisfies $\delta:=m(A)+m(B)-m(A+B)>0$, then there is an  $m'\in \mathcal M_{\tau}(\Gamma)$ such that
		\begin{enumerate}
	\item\label{item:m'kd} $m'$ is a $k\delta$-restriction of $m$;
	\item\label{item:m'stable}  $m'(A+B)=m'(A+B+K)$;
	\item\label{item:m'gap} $m'(A)+m'(B)-m'(A+B)\geq \delta.$
\end{enumerate}

\begin{proof}[Proof of Theorem \ref{th:MainOneMeanStrict}] Fix a discrete abelian group $\Gamma$ and nonempty $A, B\subseteq \Gamma$, and an invariant mean $m$ on $\ell^{\infty}(\Gamma)$ with $\delta:=m(A)+m(B)-m(A+B)>0$.   Let $K\leq \Gamma$ be the finite index subgroup given by Theorem \ref{th:OneExtreme}, and write $k$ for $[\Gamma:K]$. Let $g=1_{A}+1_{B}-1_{A+B}$.  We will find an indicator element $v\in L^{2}(m)$ satisfying
	\begin{align}
	\label{eq:mv-1} m(v)^{-1}m(1_{A+B}v)&=\dens^{*}(A+B)=\dens^{*}(A+B+K)\\
	\label{eq:mgv}	m(gv)\geq \delta. 
	\end{align}
Setting $m'(f):=m(v)^{-1}m(fv)$, we claim that properties (\ref{eq:mv-1}) and (\ref{eq:mgv}) yield conclusions \ref{item:m'kd}-\ref{item:m'gap}. 

To verify \ref{item:m'kd}, it suffices to prove $m(v)\geq k\delta$. To do so, apply Corollary \ref{cor:Q} to get $k^{-1}\geq m'(g)$. Now (\ref{eq:mgv}) implies $m'(g)\geq m(v)^{-1}\delta$.   Combining these, we get $k^{-1}\geq m(v)^{-1}\delta$, so $m(v)\geq k\delta$.

To verify \ref{item:m'stable}, note that (\ref{eq:mv-1}) implies $m'(A+B)=\dens^{*}(A+B+K)$.  The trivial inequalities $m'(A+B)\leq m'(A+B+K)\leq \dens^{*}(A+B+K)$ then imply \ref{item:m'stable}.

To verify conclusion \ref{item:m'gap}, note that (\ref{eq:mgv}) implies $m'(g)\geq m(v)^{-1}\delta$.  Since $m'(g)=m'(A)+m'(B)-m'(A+B)$ and $m'(v)\leq 1$, we get \ref{item:m'gap}.

To find an indicator element $v$ satisfying (\ref{eq:mv-1}) and (\ref{eq:mgv}),  let $\sigma$ be a Choquet decomposition of $m$ over $\mathcal M_{\tau}(\Gamma)$.  Since $\sigma$ represents $m$, we have
\begin{equation}\label{eq:DisintegrateDelta}
	\int  \nu(A)+\nu(B)-\nu(A+B)\, d\sigma(\nu) = m(A)+m(B)-m(A+B)=\delta.
\end{equation}
Let $E:=\{\nu\in \mathcal M_{\tau}^{ext}(\Gamma): \nu(A)+\nu(B)-\nu(A+B)>0\}$.  Note that Theorem \ref{th:OneExtreme} implies $\nu(A+B)=\dens^{*}(A+B)$ for all $\nu \in E$. Setting $F:=1_{E}$, (\ref{eq:DisintegrateDelta}) implies
\[
\int (\nu(A)+\nu(B)-\nu(A+B))F(\nu)\, d\sigma(\nu)\geq m(A)+m(B)-m(A+B)=\delta,
\]
while $\int \nu(A+B) F(\nu)\, d\sigma(\nu)=\dens^{*}(A+B)\int F\, d\sigma$.  Thus the hypotheses of Lemma \ref{lem:Final} are satisfied, and we get an indicator element $v$ satisfying (\ref{eq:mv-1}) and (\ref{eq:mgv}). \end{proof}

\subsection{Proof of Proposition \ref{prop:FullFolnerStrict}}

Recall the statement:

	Let $\Gamma$ be a discrete abelian group, and $A$, $B\subseteq \Gamma$ nonempty, and let $K$ be the KJ-stabilizer of $A+B$ given by Theorem \ref{th:OneExtreme}. If $\mb \Phi=(\Phi_{i})_{i\in I}$ is a \emph{full} F{\o}lner net satisfying 
\begin{equation}\label{eq:dPhiHypothesis}\delta:=\dens_{\mb \Phi}(A)+\dens_{\mb \Phi}(B)-\dens_{\mb \Phi}(A+B)>0,
\end{equation}
then there is a full F{\o}lner net $\mb \Psi=(\Psi_{i})_{i\in I}$ (with the same directed set $I$ as $\mb \Phi$)  such that $\Psi_{i}\subseteq \Phi_{i}$ for all $i\in I$, 
\begin{equation}\label{eq:dPsi}
	\liminf_{i\in I} |\Psi_{i}|/|\Phi_{i}|\geq [\Gamma:K]\delta;
\end{equation}
\begin{equation}\label{eq:GeneralKneserRepeat}
	\dens_{\mb \Psi}(A+B)=\dens_{\mb \Psi}(A+B+K)=\dens^{*}(A+B+K);
\end{equation}
and
\begin{equation}\label{eq:dAdB}\dens_{\mb \Psi}(A)+\dens_{\mb \Psi}(B) - \dens_{\mb \Psi}(A+B)\geq \delta.
\end{equation}
	
	\begin{proof}
	Let $K$ be the $KJ$-stabilizer of $A+B$ given by Theorem \ref{th:OneExtreme} and write $k$ for $[\Gamma:K]$. Let $\mb \Phi$ be a full F{\o}lner net satisfying (\ref{eq:dPhiHypothesis}).  Since $\mb \Phi$ is full, $m(f):=\lim_{i\in I} \frac{1}{|\Phi_{i}|}\sum_{\gamma\in \Phi_{i}} f(\gamma)$ defines an invariant mean on $\ell^{\infty}(\Gamma)$.  Then (\ref{eq:dPhiHypothesis}) becomes $\delta:=m(A)+m(B)-m(A+B)>0$, so we may apply Theorem \ref{th:MainOneMeanStrict} and find a mean $m'$ which is a $k\delta$-restriction of $m$, with \begin{equation}\label{eq:m'AB} m'(A+B)=m'(A+B+K)=\dens^{*}(A+B+K),
  \end{equation}
\begin{equation}\label{eq:m'Am'B}
  	m'(A)+m'(B)-m'(A+B)\geq \delta.
\end{equation} By definition, this $m'$ can be written as $m'(f):=m(v)^{-1}m(fv)$, for some indicator element $v\in L^{2}(m)$, so that $m(v)\geq k\delta$.  Lemma \ref{lem:CauchyAndInvariantSetsToFolner} then says that $m'$ can be written as $m'(f)=\lim_{i\in I}\frac{1}{|\Psi_{i}|}\sum_{\gamma\in \Psi_{i}} f(\gamma)$, for some full F{\o}lner net $\mb \Psi$ with $\Psi_{i}\subseteq \Phi_{i}$ for all $i\in I$.  The inequality $m(v)\geq k\delta$ is equivalent to $\lim_{i\in I} |\Psi_{i}|/|\Phi_{i}|\geq k\delta$, so we obtain (\ref{eq:dPsi}).  Equation (\ref{eq:GeneralKneserRepeat}) follows from (\ref{eq:m'AB}), and (\ref{eq:dAdB}) follows from (\ref{eq:m'Am'B}).  
	\end{proof}

\section{Proof of Lemma \ref{lem:LevelSetInL2}}\label{sec:LevelProof}
	
	The following is known as Helly's selection lemma.  See \cite[p.~222]{Natanson_TheoryOfFunctions} for a proof.
	
	\begin{lemma}\label{lem:Helly}
		If $I\subseteq \mathbb R$ is an interval and $\phi_{n}:I\to \mathbb [0,1]$ is a sequence of increasing functions, then there is a subsequence $(\phi_{n_{k}})_{k\in \mathbb N}$ converging pointwise to an increasing function $\phi:I\to [0,1]$.
	\end{lemma}
	The proof of the next lemma will use the triangle inequality for finitely additive measures: $m(A\triangle B)\leq m(A\triangle C)+m(C\triangle B)$, and the identity
	\begin{equation}\label{eq:SymDiff}
		m(A\triangle B) =  2m(A\setminus B)+m(B)-m(A).
	\end{equation}
	
		Lemma \ref{lem:LevelSetInL2} follows immediately from the next lemma. 
	
\begin{lemma}	Let $\Gamma$ be a discrete abelian group and $m\in \mathcal M(\Gamma)$. Let $w\in L^{2}(m)$ be represented by an $L^{2}(m)$-Cauchy sequence $f_{n}: \Gamma \to [0,1]$. 
	
Let $C_{n}(x):=f_{n}^{-1}([0,x])$, and  $\phi_{n}(x):=m(C_{n}(x))$. Then there is a countable set $\Omega\subseteq [0,1]$ such that
\begin{enumerate}
	\item\label{item:phiExists} $\phi(x):=\lim_{n\to\infty} \phi_{n}(x)$ exists and is continuous at every $x\in [0,1]\setminus \Omega$.  

\item\label{item:Ccauchy} For all $b\in [0,1]\setminus \Omega$, we have
\begin{equation}\label{eq:SimpleCauchy}
	\lim_{N\to \infty} \sup_{n,n'\geq N} m(C_{n}(b)\triangle C_{n'}(b))=0.
\end{equation}
\item\label{item:ifwInvariant} If $w$ is $\tau$-invariant, we also have
\begin{equation}\label{eq:TauLevel}
	\lim_{n\to\infty} m(C_{n}(b)\triangle (C_{n}(b)+\gamma))=0 \quad \text{ for all } b\in [0,1]\setminus \Omega.
\end{equation}
\end{enumerate}
Furthermore, for all $b\in [0,1]\setminus \Omega$, setting $D_{n}(b):=f_{n}^{-1}[0,b)$ we have 
\begin{equation}\label{eq:EndPt} 
	\lim_{n\to\infty} m(C_{n}(b)\triangle  D_{n}(b))=0
\end{equation}
The set $\Omega$ depends only on $w$, and not on the choice of $(f_{n})_{n\in \mathbb N}$ representing $w$.  
\end{lemma}
	\begin{proof} Let $w$, $f_{n}$, and $C_{n}$ be as in the hypothesis.
				\begin{claim}\label{claim:Delta0}
			For all $b\in \mathbb [0,1]$ and all $\delta>0$,
			\begin{equation}\label{eq:CabEps}
				\lim_{N\to\infty} \sup_{n,n'\geq N} m(C_{n}(b)\setminus C_{n'}(b+\delta))= 0.
			\end{equation}
		\end{claim}
		\begin{proof}[Proof of Claim.] Suppose otherwise, and let $E(n,n'):=C_{n}(b)\setminus C_{n'}(b+\delta)$.  For all $\gamma\in E(n,n')$, we have $f_{n}(\gamma)\in [a,b]$ and $f_{n'}(\gamma)\notin [0,b+\delta]$. Thus $|f_{n}(\gamma)- f_{n'}(\gamma)|\geq \delta$ for all such $\gamma$, and we have $m(|f_{n}-f_{n'}|)\geq \delta m(E(n,n'))$.  Since $(f_{n})_{n\in \mathbb N}$ is $L^{2}(m)$-Cauchy, the left-hand side of the latter inequality converges to $0$ as $\min(n,n')\to\infty$. This proves the claim. \end{proof}
		
		Note that the Claim implies that for all $a<b\leq 1$, we have $\limsup_{n\to\infty} \phi_{n}(a)\leq \liminf_{n\to\infty} \phi_{n}(b)$.
		
		Since each $\phi_{n}$ is increasing, Lemma \ref{lem:Helly} provides a subsequence $(\phi_{n_{k}})$ converging pointwise to an increasing function $\phi$.  We now show that if $\phi$ is continuous at $a$, then $\lim_{n\to\infty} \phi_{n}(a)$ exists.  For such $a$, we have \[\limsup_{n\to\infty} \phi_{n}(a)\leq \liminf_{n\to\infty} \phi_{n}(a+\delta)\leq \phi(a+\delta),\] and $\liminf_{n\to\infty} \phi_{n}(a)\geq \limsup_{n\to\infty} \phi_{n}(a-\delta)\geq \phi(a-\delta)$.  By continuity of $\phi$ at $a$, this implies $\limsup_{n\to\infty} \phi_{n}(a)=\liminf_{n\to\infty} \phi_{n}(a)$.		
		
		We now prove (\ref{eq:SimpleCauchy}), assuming $\phi$ is continuous at $b$.  Let $\varepsilon>0$.  By continuity, choose $\delta>0$ so that $\lim_{n\to\infty} \phi_{n}(b+\delta)$ exists and $\phi(b+\delta)<\phi(b)+\varepsilon$.  Choose $N_{1}$ large enough that $|\phi_{n}(x)-\phi(x)|<\varepsilon$ for all $n\geq N$ and all $x\in \{b,b+\delta\}$.  By the triangle inequality we have $\phi_{n}(b+\delta)<\phi_{n}(b)+2\varepsilon$.
		
		By the Claim, choose $N_{2}$ large enough that  $m(C_{n}(b)\setminus  C_{n'}(b+\delta))<\varepsilon$ for all $n, n'\geq N_{2}$.  Let $N=\max\{N_{1},N_{2}\}$.
		
		Let $n, n'\geq N$.  Then $m(C_{n'}(b+\delta))<m(C_{n'}(b))+2\varepsilon$, so $m(C_{n'}(b)\triangle C_{n'}(b+\delta))<2\varepsilon$.  This implies $|m(C_{n}(b)\setminus C_{n'}(b))- m(C_{n}(b)\setminus C_{n'}(b+\delta))|<2\varepsilon$. Then the Claim implies $m(C_{n}(b)\setminus C_{n'}(b))<2\varepsilon$.  By symmetry we get $m(C_{n'}(b)\setminus C_{n}(b))<2\varepsilon$. This proves (\ref{eq:SimpleCauchy}).
		
		To prove (\ref{eq:TauLevel}), assume $w$ is $\tau$-invariant, let $\varepsilon>0$, and let $\gamma \in \Gamma$. Fix $b\in [0,1]\setminus \Omega$.  Then $C_{n}(b)-\gamma = (\tau_{\gamma}f_{n})^{-1}([0,b])$.  We claim that for all $\delta>0$, 
		\begin{equation}\label{eq:Cnb}
			\lim_{n\to\infty} m((C_{n}(b)-\gamma)\setminus C_{n}(b+\delta))=0.
		\end{equation}
		This follows from the estimate $m(|f_{n}-\tau_{\gamma}f_{n}|)\geq \delta m((C_{n}(b)-\gamma)\setminus C_{n}(b+\delta))$.  
		
		Choose $\delta$ small enough and $N$ large enough  that $m(C_{n}(b+\delta)\triangle C_{n}(b))<\varepsilon$ for every $n\geq N$.  Then \[m((C_{n}(b)-\gamma)\setminus C_{n}(b))-m((C_{n}+\gamma)\setminus C_{n}(b+\delta))< \varepsilon.\]
		Together with (\ref{eq:Cnb}), this implies $m((C_{n}(b)-\gamma)\setminus C_{n}(b))<2\varepsilon$ for large enough $n.$  Likewise we have $m(C_{n}(b)\setminus (C_{n}(b)-\gamma))$ for large enough $n$.  This proves (\ref{eq:TauLevel}).
		
		To prove (\ref{eq:EndPt}), note that for all $\delta>0$, $m(C_{n}(b)\triangle D_{n}(b))\leq \phi(b+\delta)-\phi(b-\delta)$.  Continuity now implies (\ref{eq:EndPt}).
		
		To prove that $\Omega$ depends only on $w$, suppose $(f_{n})_{n\in \mathbb N}$ and $(g_{n})_{n\in \mathbb N}$ are two $L^{2}(m)$-Cauchy sequences representing $w$.  The sequence $(h_{n})_{n\in \mathbb N}=(f_{1},g_{1},f_{2},g_{2},\dots)$ is also Cauchy.  If we define $\psi_{n}(x):=m(h_{n}^{-1}[0,x])$, we see that $\psi_{2n-1}$ is the sequence $\phi_{n}$ defined above.  As above, we have $\psi(x):=\lim_{n\to\infty} \psi_{n}(x)$ exists for all but countably many $x$, and is increasing.  Since $(\phi_{n})_{n\in\mathbb N}$ is a subsequence of $(\psi_{n})_{n\in \mathbb N}$, we have that $\phi$ and $\psi$ both exist at every $x$ where $\phi$ is continuous.  Thus (\ref{eq:SimpleCauchy}) holds with $(g_{n})_{n\in \mathbb N}$ in place of $(f_{n})_{n\in \mathbb N}$, at every $b$ where $\phi$ is continuous.
		  \end{proof}

\printbibliography
\end{document}